\documentclass{amsart}
\usepackage[numbers]{natbib}

\usepackage{hyperref}
\usepackage{graphicx}
\usepackage[english,ngerman]{babel}
\usepackage{amssymb}
\usepackage{bbm}
\usepackage{amsmath}
\usepackage{amsthm}
\usepackage{cases}
\usepackage{color}

\numberwithin{equation}{section}

\newcommand{\titlestring}{Complex Random Energy Model: Zeros and Fluctuations}

\newcommand{\subjectstring}{Primary: 60G50; Secondary: 82B44, 60E07, 30B20, 60F05, 60F17, 60G15}

\newcommand{\keywordsstring}{random energy model, sums of random exponentials,
zeros of random analytic functions, central limit theorem, extreme value
theory, stable distributions, logarithmic potentials}

\renewcommand{\P}{\mathbb {P}}
\newcommand{\E}{\mathbb {E}}
\newcommand{\T}{\mathbb {T}}
\newcommand{\GGG}{\mathbb {G}}
\newcommand{\WWW}{\mathbf {W}}
\newcommand{\ZZZ}{\mathcal {Z}}

\newcommand{\MMM}{\mathbb {M}}
\newcommand{\C}{\mathbb {C}}
\newcommand{\R}{\mathbb {R}}
\newcommand{\Z}{\mathbb {Z}}
\newcommand{\N}{\mathbb {N}}
\newcommand{\dd}{{\rm d}}
\newcommand{\nn}{{\textbf n}}
\newcommand{\ee}{{\rm e}}

\DeclareMathOperator{\var}{Var}
\DeclareMathOperator{\Var}{Var}

\DeclareMathOperator{\corr}{Corr}

\renewcommand{\Re}{\operatorname{Re}}
\renewcommand{\Im}{\operatorname{Im}}

\newcommand{\eps}{\varepsilon}

\newcommand{\eqdistr}{\stackrel{{d}}{=}}
\newcommand{\todistr}{\toweak}
\newcommand{\todistrnoN}{\toweaknoN}
\newcommand{\toprobabnoN}{\overset{{P}}{\underset{}\longrightarrow}}
\newcommand{\toweaknoN}{\overset{{w}}{\underset{}\longrightarrow}}
\newcommand{\toweak}{\overset{{w}}{\underset{N\to\infty}\longrightarrow}}
\newcommand{\toprobab}{\overset{{P}}{\underset{N\to\infty}\longrightarrow}}

\newcommand{\bsl}{\backslash}
\newcommand{\ind}{\mathbbm{1}}

\newtheorem{theorem}{Theorem}[section]

\newtheorem{corollary}[theorem]{Corollary}

\newtheorem{lemma}[theorem]{Lemma}
\newtheorem{proposition}[theorem]{Proposition}

\theoremstyle{remark}
\newtheorem{remark}[theorem]{Remark}

\ifx\hypersetup\undefined\else
\fi

\begin{document}

\selectlanguage{english}

\begin{center}
\LARGE{\titlestring}
\end{center}
\vskip1cm
\noindent{\large Zakhar~Kabluchko}
\\
\textit{
Institute of Stochastics, Ulm University, Germany
}
\\
e-mail: \texttt{zakhar.kabluchko@uni-ulm.de}
\vskip0.5cm
\noindent{\large Anton~Klimovsky}\footnote{Research supported in part by the European
Commission (Marie Curie fellowship, project PIEF-GA-2009-251200).}
\\
\textit{
Mathematical Institute, Leiden University, The Netherlands}
\\
e-mail: \texttt{ak@aklimovsky.net}
\vskip0.5cm
\begin{center}
\today
\end{center}
\vskip0.5cm \noindent \textbf{AMS 2000 Subject Classification:} \subjectstring.
\vskip0.5cm
\begin{center}
\textbf{Abstract}
\end{center}

\noindent The partition function of the random energy model at inverse temperature $\beta$
is a sum of random exponentials $ \ZZZ_N(\beta)=\sum_{k=1}^N \exp(\beta \sqrt{n}
X_k)$, where $X_1,X_2,\ldots$ are independent real standard normal random
variables (= random energies), and $n=\log N$. We study the large $N$ limit of
the partition function viewed as an analytic function of the complex variable
$\beta$. We identify the asymptotic structure of complex zeros of the partition
function confirming and extending predictions made in the theoretical physics
literature. We prove limit theorems for the random partition function at complex
$\beta$, both on the logarithmic scale and on the level of limiting
distributions. Our results cover also the case of the sums of independent
identically distributed  random exponentials with any given correlations between
the real and imaginary parts of the random exponent.

\section{Introduction}

Since the pioneering work of Lee and Yang~\cite{lee_yang1,lee_yang2}, much
attention in the statistical physics literature has been paid to studying
partition functions of various models at \textit{complex} values of parameters
such as complex inverse temperature or complex external magnetic
field, see, e.g.,~\cite{BiskupEtAl2000,BenaEtAl2005}. These studies are
sometimes referred to as the Lee--Yang program. The motivation here is to
identify the mechanisms causing phase transitions of the model under study.
These transitions manifest themselves in the analyticity breaking of the
logarithm of the partition function which, in turn, is related to the
\emph{complex zeros} of the partition function. Phase transitions are thus
associated with the  accumulation points  of the complex zeros of the partition
function on the real axis, in the large system limit. In this respect,
complex-valued parameters provide a clean framework for
identification of phase transitions.

The main emphasis of Lee--Yang program was on the classical lattice
models of statistical mechanics. In this work, we focus on the simplest model
of a spin glass~\cite{talagrand_book2,BovierBook2006}: the \textit{random energy
model} (REM) introduced by Derrida~\cite{Derrida1980,Derrida1981}. Let
$X,X_1,X_2,\ldots$ be independent real standard normal random variables. 
The partition function of
REM at inverse temperature $\beta$ is defined by
\begin{equation}\label{eq:def_ZZZ_N}
\ZZZ_N(\beta)=\sum_{k=1}^N \ee^{\beta \sqrt n X_k}.
\end{equation}
Here, $N$ is a large integer, and we use the notation $n=\log N$. For
\textit{real} inverse temperature $\beta>0$ the asymptotic behavior of
$\ZZZ_N(\beta)$ as $N \to \infty$ (or equivalently, as $n \to \infty$) has been
extensively studied in the literature;
see~\cite{eisele,olivieri_picco,bovier_kurkova_loewe,
ben_arous_bogachev_molchanov,BovierBook2006}. Specifically, the limiting log-partition function
is given by the formula
\begin{equation}\label{eq:limiting_log_partition_function_real}
p(\beta)
:=\lim_{N \to \infty}
\frac{1}{n}
\log | \ZZZ_N(\beta)|=
\begin{cases}
1+\frac{1}{2}\beta^2,
&0\leq \beta\leq \sqrt2,\\
\sqrt 2\beta,
&\beta\geq \sqrt 2.
\end{cases}
\end{equation}
Convergence \eqref{eq:limiting_log_partition_function_real} holds both a.s.\ and in $L^q$, $q\geq 1$; see~\cite{eisele,olivieri_picco,BovierBook2006}. On the level of fluctuations, it
has been shown in~\cite{bovier_kurkova_loewe} that  upon suitable rescaling, the
random variable $\ZZZ_N(\beta)$ becomes asymptotically Gaussian, for $\beta \leq
\sqrt{2}/2$, whereas, for $\beta > \sqrt{2}/2$, it has limiting stable
non-Gaussian distribution, as $N \to \infty$.

\begin{figure}[htbp]
\includegraphics[width=0.55\textwidth]{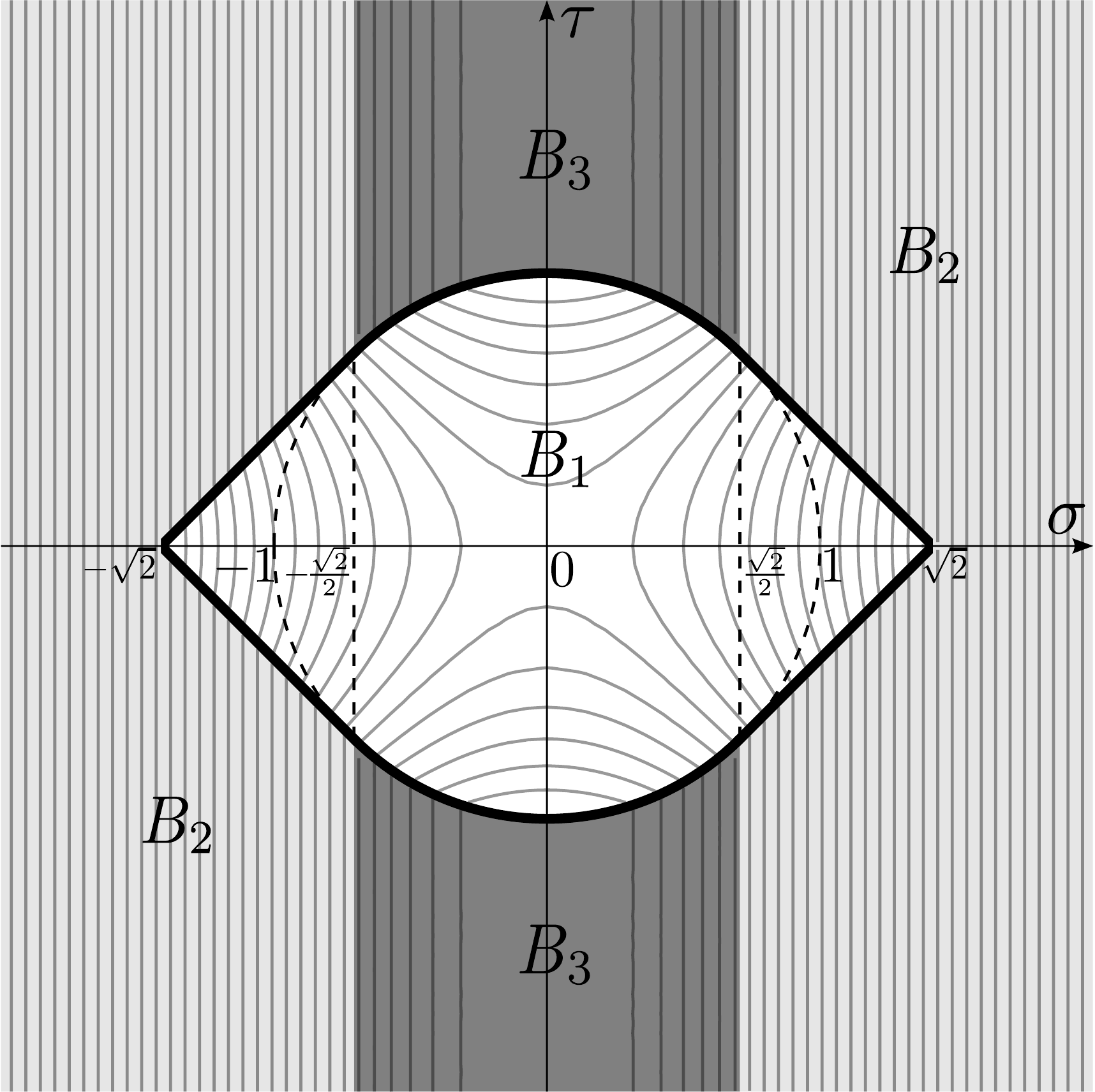}
\caption
{\small
Complex zeros of $\ZZZ_N$ in the large $N$ limit. There are three phases: $B_1$ (white, no zeros), $B_2$ (light gray, density of zeros is of order $1$), $B_3$ (dark gray, density of zeros is of order $n$). On the boundary of $B_1$ the linear density of zeros is of order $n$. The plot shows also the contour lines (gray curves and lines) of the log-partition function $p$.
} \label{fig}
\end{figure}

Using heuristic
arguments,
\citet{DerridaComplexREM1991}
studied the REM at \textit{complex} inverse temperature $\beta=\sigma+i\tau$.
He derived the following logarithmic asymptotics extending~\eqref{eq:limiting_log_partition_function_real} to the complex plane:
\begin{equation}\label{eq:limiting_log_partition_function}
p(\beta)
:=\lim_{N \to \infty}
\frac{1}{n}
\log | \ZZZ_N(\beta)|=
\begin{cases}
1+\frac{1}{2}(\sigma^2 - \tau^2),
&\beta \in \overline{B}_{1},\\
\sqrt{2} |\sigma|,
&\beta \in \overline{B}_{2},\\
\frac 12 +\sigma^2,
&\beta \in \overline{B}_{3},
\end{cases}
\end{equation}
where $B_{1}$, $B_{2}$, $B_{3}$ are three subsets of the complex plane (see Figure~\ref{fig}) defined by
\begin{align}
B_{1}
&=
\C\bsl \overline{B_2\cup B_3},
\label{eq:phases1}\\
B_{2}
&=
\{\beta\in\R^2 \colon 2\sigma^2 >  1, |\sigma|+|\tau| > \sqrt {2}\},
\label{eq:phases2}\\
B_{3}
&=
\{ \beta \in \R^2 \colon 2\sigma^2 < 1, \sigma^2+\tau^2> 1\}.
\label{eq:phases3}
\end{align}
Here,  $\bar{A}$ denotes the closure of the set $A$. Note that the limiting log-partition function $p$ is  continuous.

To derive~\eqref{eq:limiting_log_partition_function}, \citet{DerridaComplexREM1991} used an approach which can be roughly described as follows.
Instead of $\ZZZ_N(\beta)$, one can consider the truncated sum
$$
\ZZZ_N^{*}(\beta)=\sum_{k=1}^N \ee^{\beta \sqrt n X_k}\ind_{|X_k|<\sqrt{2n}}.
$$
Indeed, with high probability it holds that $\ZZZ_N^{*}(\beta)=\ZZZ_N(\beta)$
since the order of the maximum term among $|X_1|,\ldots,|X_N|$ is $\sqrt{2n}$
and the existence of an outlier satisfying $|X_k|>\sqrt{2n}$ has probability
converging to $0$. Note, however, that although $\ZZZ_N^{*}(\beta)$ and
$\ZZZ_N(\beta)$ are close in probability, their expectations (and standard
deviations) may be very different from each other, at least for some values of
$\beta$. Derrida derived an asymptotic formula for the expectation of
$\ZZZ_N^{*}(\beta)$, as $N\to\infty$, using the saddle-point method. Two cases
are possible: the expectation is dominated by the energies $X_k$ inside the
interval $(-\sqrt {2n}, \sqrt {2n})$ (equivalently, the contribution of the
saddle point dominates the expectation), or by the energies located near one of
the boundary points $\pm\sqrt{2n}$.  He also obtained two similar cases for the
standard deviation of $\ZZZ_N^{*}(\beta)$. Comparing the resulting four
formulas, Derrida discovered the three phases $B_1,B_2,B_3$. The arguments
of~\citet{DerridaComplexREM1991} are not fully rigorous, although it should be
emphasized that he did not use the replica method or other standard non-rigorous
spin glass method. In the present paper, we will make the argument of Derrida
rigorous and refine his results by deriving distributional limit theorems for
the fluctuations of $\ZZZ_N(\beta)$ (and for the fluctuations in some more
general models, see~Section~\ref{subsec:fluct}) at complex $\beta$. An essential
feature of the REM at complex temperature is the possibility of canceling of
terms in $\ZZZ_N(\beta)$ due to the presence of complex amplitudes. It is for
this reason that some standard techniques of rigorous spin glass
theory~\cite{talagrand_book2} like the concentration inequalities or the
second-moment method do not (or do not always) lead to the desired result. These
difficulties will be discussed in more detail in
Section~\ref{subsec:proof_theo_log_scale}.

Based on his formula~\eqref{eq:limiting_log_partition_function} for the limiting
log-partition function, \citet{DerridaComplexREM1991} computed the asymptotic
distribution of zeros of $\ZZZ_N$ in the complex plane. His predictions were in
a good agreement with the numerical simulations of~\citet{moukarzel_parga}.
Derrida observed that since
 $\ZZZ_N(\beta)$ is an analytic function of
$\beta$, its empirical distribution of zeros (a measure assigning to every zero
weight $1$, with multiplicities) is given by $\frac 1 {2\pi} \Delta
\log|\ZZZ_N|$, where $\Delta=\frac{\partial^2}{\partial
\sigma^2}+\frac{\partial^2}{\partial \tau^2} $ denotes the Laplace operator in
the $\beta$-plane. Taking the large $N$ limit, Derrida obtained the formula
$\frac {n}{2\pi} \Delta p$ for the asymptotic distribution of zeros of $\ZZZ_N$.
Since the function $p$ is harmonic in $B_1$ and $B_2$, Derrida predicted that
``there should be no zeros (or at least the density of zeros vanishes) in phases
$B_1$ and $B_2$''. In phase $B_3$,  ``the density of zeros is uniform'' and is
asymptotic to $\frac{n}{2\pi}$.  Also, since the normal derivative of $p$ has a
jump on the boundary of $B_1$, but has no jump on the boundary between $B_1$ and
$B_3$ ``the boundaries between phases $B_1$ and $B_2$, and between phases $B_1$
and $B_3$ are lines of zeros whereas the separation between phases $B_2$ and
$B_3$ is not''. The argument of Derrida involves interchanging the Laplace
operator and the large $N$ limit. In the present paper we justify Derrida's
approach rigorously and derive further results on the distribution of zeros of
$\ZZZ_N$. Namely, we relate the zeros of $\ZZZ_N$ to the zeros of two random
analytic functions: a Gaussian analytic function $\GGG$ (in phase $B_3$), and a
zeta-function $\zeta_P$ associated to the Poisson process (in phase $B_2$).
Also, we will clarify the local structure of the mysterious ``lines of zeros''
on the boundary of $B_1$.

For the partition function of REM, considered as a function of a complex
external magnetic field, a non-rigorous analysis similar to that
of~\citet{DerridaComplexREM1991} has been carried out
by~\citet{moukarzel_parga_analytic,moukarzel_parga_magnetic}.  For directed
polymers with complex weights on a tree, which is another related model, the
logarithmic asymptotics~\eqref{eq:limiting_log_partition_function} has been
derived in~\cite{DerridaEvansSpeer1993}; see also~\cite{biggins,koukiou}.
Recently, \citet{takahashi} and~\citet{obuchi_takahashi} studied the complex
zeros in the generalized REM and other spin glass models using the non-rigorous
replica method.  However, spin glasses at complex temperature have not been much
studied rigorously in the mathematics literature. Our aim is to fill this gap.

Substantial motivation for the setup with complex random energies comes
from quantum mechanics. There, the sums of random exponentials with
complex-valued exponents arise naturally in the models of interference in
inhomogeneous media~\cite{DerridaEvansSpeer1993,DobrinevskiLeDoussalWiese2011},
and in the studies of the quantum Monte~Carlo method~\cite{DuringKurchan2010}.

The sum of random exponentials $\ZZZ_N$ is a natural random analytic function
exhibiting, despite of its simple form, a rather non-trivial behavior. We hope
that the methods developed to study this function can be applied to other random
analytic functions, for example to random polynomials or random Taylor series.
For a recent work in this direction, we refer
to~\cite{kabluchko_zaporozhets,kabluchko12}. Also,  $\ZZZ_N$ can be interpreted
as a (normalized) characteristic function of the i.i.d.\ normal sample
$X_1,\ldots,X_N$. This connection will be discussed in
Section~\ref{sec:extensions}.

The paper is organized as follows.  After introducing some notation in
Section~\ref{subsec:notation}, we state our results on zeros and fluctuations in
Sections~\ref{subsec:zeros} and~\ref{subsec:fluct}, respectively.  Proofs can be
found in Sections~\ref{sec:proofs_fluct} and~\ref{sec:proofs_zeros}. In
Section~\ref{sec:extensions}, we discuss possible extensions and open problems
related to our results.

\section{Statement of results}\label{sec:results}
\subsection{Notation}\label{subsec:notation}
We will write the complex inverse temperature $\beta$ in the
form $\beta=\sigma+i\tau$, where $\sigma,\tau\in\R$. We use the notation $n=\log
N$, where $N$ is a large integer and the logarithm is natural. Note that in the
physics literature on the REM, it is customary to take the logarithm at basis
$2$. Replacing $\beta$ by $\beta/\sqrt {\log 2}$ in our results we can easily
switch to the physics notation.

We denote by $N_{\R}(0,s^2)$ the real Gaussian
distribution with mean zero and variance $s^2>0$. By $N_{\C}(0,s^2)$, we denote the complex Gaussian distribution with density
$$
z\mapsto \frac 1{\pi s^2} \ee^{-|z/s|^2}
$$
w.r.t.\ the Lebesgue measure on $\C$. Note that $Z\sim N_{\C}(0,s^2)$ iff $Z=X+iY$, where $X,Y\sim
N_{\R}(0,s^2/2)$ are independent. In this case, $\E Z=0$ and $\E |Z|^2=1$.  Real or complex normal distribution is referred to as standard if
$s=1$. The standard normal distribution function is denoted by $\Phi$.

Convergence in probability and weak (distributional) convergence will be denoted by $\toprobabnoN$ and $\toweaknoN$, respectively. Let $C$ be a generic positive constant whose value will change at different occurrences.

\subsection{Results on zeros}\label{subsec:zeros}
Let $\ZZZ_N$ be the partition function of the REM defined as in~\eqref{eq:def_ZZZ_N}. Note the distributional equalities
\begin{align}
\label{eq:symmetries}
\ZZZ_N(\beta)\eqdistr\ZZZ_N(-\beta), \;\;\;
\ZZZ_N(\overline{\beta})\eqdistr\overline{\ZZZ_N(\beta)}.
\end{align}
Due to \eqref{eq:symmetries}, it is often enough to consider the case $\sigma, \tau \geq 0$.
The next result describes the global structure of complex zeros of $\ZZZ_N$, as
$N\to\infty$.
Let $\Xi_3$ be the Lebesgue measure restricted to $B_3$. Also, let $\Xi_{13}$ be
the one-dimensional length measure on the boundary between $B_1$ and $B_3$
(which consists of two circular arcs). Finally, let $\Xi_{12}$ be a measure
having the density $\sqrt 2 |\tau|$ with respect to the one-dimensional length
measure restricted to the boundary between  $B_1$ and $B_2$ (which consists of
four line segments). Define a measure $\Xi=2\Xi_3+\Xi_{12}+\Xi_{13}$.

\begin{theorem}\label{theo:zeros_in_B3}
For every continuous function $f\colon\C\to \R$ with compact support,
\begin{equation}\label{eq:theo:zeros_in_B3}
\frac{1}{n} \sum_{\beta\in \C \colon \ZZZ_N(\beta)=0} f(\beta)
\toprobab
\frac{1}{2\pi}\int_{\C} f(\beta) \Xi(\dd \beta).
\end{equation}
\end{theorem}
\begin{remark}
As a consequence, the random measure assigning a weight $1/n$ to each zero of
$\ZZZ_N$ converges weakly to the deterministic measure $\frac 1 {2\pi} \Xi$. The
limit measure $\Xi$ is related to the limiting log-partition function $p$,
see~\eqref{eq:limiting_log_partition_function}, by the formula $\Xi=\Delta p$,
in accordance with~\cite{DerridaComplexREM1991}. Here, $\Delta$ is the Laplace
operator which should be understood in the \textit{distributional} sense. The
\textit{point-wise} Laplacian of $p$ is easily seen to be $2 \Xi_3$. However, in
the distributional Laplacian there are additional terms which come from the
jumps of the normal derivative of
$p$ along the boundaries $\bar B_1\cap \bar B_2$ and $\bar B_1\cap \bar B_3$. On
the boundary $\bar B_2\cap \bar B_3$ the jump turns out to be $0$. 
that $p$ can be viewed as the two-dimensional electrostatic potential generated
by the charge distribution $\Xi$.
\end{remark}

Theorem~\ref{theo:zeros_in_B3} makes the last formula
in~\cite{DerridaComplexREM1991} rigorous. In the next theorems, we will
investigate more fine properties of the zeros of $\ZZZ_N$.  We start by
describing  the local structure of zeros of $\ZZZ_N$ in a neighborhood of area
$1/n$ of a fixed point $\beta_0\in B_3$.  Let $\{\GGG(t) \colon t\in\C\}$ be a 
Gaussian  random analytic function~\cite{nazarov_sodin} given by
\begin{equation}\label{eq:def_GGG}
\GGG(t)=\sum_{k=0}^{\infty} \xi_k \frac{t^k}{\sqrt{k!}},
\end{equation}
where $\xi_0,\xi_1,\ldots$ are independent standard complex Gaussian random
variables. The complex zeros of $\GGG$ form a remarkable point process which has
intensity $1/\pi$  and is translation invariant. Up to rescaling, this is the
only translation invariant zero set of a Gaussian analytic function;
see~\cite[Section~2.5]{peres_etal_book}. This and related zero sets have been
much studied;  see the monograph~\cite{peres_etal_book}.

\begin{theorem}\label{theo:local_zeros_small_beta}
Let $\beta_0\in B_3$ be fixed. For every continuous function $f\colon \C\to\R$
with compact support,
$$
\sum_{\beta\in \C \colon \ZZZ_N(\beta)=0} f(\sqrt n (\beta - \beta_0))
\todistr
\sum_{\beta\in\C \colon \GGG(\beta)=0} f(\beta).
$$
\end{theorem}
\begin{remark}
Equivalently, the point process consisting of the points $\sqrt n
(\beta-\beta_0)$, where $\beta$ is a zero of $\ZZZ_N$, converges weakly to the
point process of zeros of $\GGG$.
\end{remark}

\citet{DerridaComplexREM1991} predicted that the set $B_1$ should be free of
zeros. As we will see below, it is not true that the number of zeros in $B_1$
converges to $0$ in probability since with non-vanishing probability there exist
zeros very close to the boundary of $B_1$. However, a slightly weaker statement
is true.
\begin{theorem}\label{theo:no_zeros_B1}
Let $K$ be a compact subset of $B_1$. Then, there exists $\eps>0$ depending on $K$ such that
\begin{equation*}
\P[\ZZZ_N(\beta)=0, \text{ for some } \beta\in K]  =O(N^{-\eps}),\;\;\; N\to\infty.
\end{equation*}
\end{theorem}
As a consequence, the number of zeros of $\ZZZ_N$ in $K$ converges to $0$ in
probability. It is natural to conjecture that the convergence holds a.s. The
number $\eps$, as provided by the proof of Theorem~\ref{theo:no_zeros_B1},
converges to $0$ as the distance between $K$ and the boundary of $B_1$ gets
smaller. So, the a.s.\ convergence does not follow from a Borel--Cantelli
argument.

Consider now the zeros of $\ZZZ_N$ in the set $B_2$.  We will show that in the
limit as $N\to\infty$ the zeros of $\ZZZ_N$ in $B_2$ look like  the zeros of
certain random analytic function $\zeta_{P}$. This function may be viewed as a
zeta-function associated to the Poisson process. It is defined as follows. Let
$P_1<P_2<\ldots$ be the arrival times of a unit intensity homogeneous Poisson
process on the positive half-line. That is, $P_k=\eps_1+\ldots+\eps_k$, where
$\eps_1,\eps_2,\ldots$ are i.i.d.\ standard exponential random variables, i.e., $\P[\eps_k>t]=\ee^{-t}$,
$t\geq 0$. For $T>1$, define the random process
\begin{equation}\label{eq:zeta_poi_tilde}
\tilde \zeta_P(\beta; T)=\sum_{k=1}^{\infty} \frac{1}{P_k^{\beta}}\ind_{P_k\in [0,T]} -\int_1^T t^{-\beta}\dd t,
\;\;\; \beta\in\C.
\end{equation}
\begin{theorem}\label{theo:zeta_poi_exists}
With probability $1$,  the sequence $\tilde \zeta_P(\beta; T)$ converges as
$T\to\infty$  to a limiting function denoted by $\tilde \zeta_{P}(\beta)$. The
convergence is uniform on compact subsets of the half-plane $\{\beta\in \C \colon \Re
\beta>1/2\}$.
\end{theorem}
\begin{corollary}\label{cor:zeta_poi_exists}
With probability $1$, the Poisson process zeta-function
\begin{equation}\label{eq:zeta_poi}
\zeta_{P}(\beta)=\sum_{k=1}^{\infty} \frac{1}{P_k^{\beta}}
\end{equation}
defined originally for $\Re \beta>1$, admits a meromorphic continuation to the
domain $\Re \beta>1/2$. The function $\tilde
\zeta_{P}(\beta)=\zeta_{P}(\beta)-\frac 1 {\beta-1}$ is a.s.~analytic in this
domain.
\end{corollary}
The next theorem describes the limiting structure of zeros of $\ZZZ_N$ in $B_2$.
The form of the process $\zeta_P$ appearing there is not surprising and can be
explained as follows.  In phase $B_2$ the process $\ZZZ_N$ is dominated by the
extremal order statistics of the sample $X_1,\ldots,X_N$. These form a Poisson
point process in the large $N$ limit, see, e.g.,
\cite[Corollary~4.19(i)]{resnick_book}, and $\zeta_P$ is some functional of this
process.
\begin{theorem}\label{theo:local_zeros_large_beta}
Let $f \colon B_2\to\R$ be a continuous function with compact support. Let
$\zeta_P^{(1)}$ and $\zeta_P^{(2)}$ be two independent copies of $\zeta_P$.
Then,
$$
\sum_{\beta\in B_2 \colon \ZZZ_N(\beta)=0} f(\beta)
\todistr
\sum_{\substack{\beta\in B_2 \colon \\ \zeta_{P}^{(1)}\left(\beta / \sqrt 2\right)=0}} f(\beta)
+
\sum_{\substack{\beta\in B_2 \colon \\ \zeta_{P}^{(2)}\left(\beta / \sqrt 2 \right)=0}} f(-\beta).
$$
\end{theorem}
Theorem~\ref{theo:local_zeros_large_beta} tells us that the zeros of $\ZZZ_N$ in
the domain $\sigma>1/\sqrt 2, |\sigma|+|\tau|>\sqrt 2$ (which constitutes one
half of $B_2$) have approximately the same law as the zeros of $\zeta_P$, as
$N\to\infty$.  Let us stress that the approximation breaks down in the triangle
$\sigma>1/\sqrt 2$, $|\sigma|+|\tau|<\sqrt 2$. Although the function $\zeta_P$
is well-defined and may have zeros there, the function $\ZZZ_N$ has, with high
probability, no zeros in any compact subset of the triangle by
Theorem~\ref{theo:no_zeros_B1}.

Next we state some properties of the function $\zeta_P$. Let  $\beta>1/2$ be
real. For $\beta\neq 1$, the  random variable $\zeta_{P}(\beta)$ is stable with
index $1/\beta$ and skewness parameter $1$. In fact, \eqref{eq:zeta_poi_tilde}
is just the series representation of this random variable;
see~\cite[Theorem~1.4.5]{samorodnitsky_taqqu_book}. For $\beta=1$, the random
variable $\tilde \zeta_{P}(1)$ (which is the residue of $\zeta_{P}$ at $1$) is
$1$-stable with skewness $1$. For general complex $\beta$, we have the following
stability property.
\begin{proposition}
If $\zeta_{P}^{(1)},\ldots, \zeta_{P}^{(k)}$ are independent copies of
$\zeta_{P}$, then we have the following distributional equality of stochastic
processes:
$$
\zeta_{P}^{(1)} + \ldots + \zeta_{P}^{(k)}\eqdistr k^{\beta}\zeta_{P}.
$$
\end{proposition}
To see this, observe that the union of $k$ independent unit intensity Poisson
processes has the same law as a single unit intensity Poisson process scaled by
the factor $1/k$. As a corollary, the distribution of the random vector $(\Re
\zeta_{P}(\beta), \Im \zeta_{P}(\beta))$ belongs to the family of operator
stable laws; see~\cite{meerschaert_book}.

\begin{proposition}\label{prop:zeta_boundary}
Fix $\tau\in\R$. As $\sigma\downarrow 1/2$, we have
$$
\sqrt{2\sigma-1}\, \zeta_P(\sigma + i\tau) \todistrnoN
\begin{cases}
N_{\C}(0,1), & \text{ if } \tau\neq 0,\\
N_{\R}(0,1), & \text{ if } \tau=0.
\end{cases}
$$
\end{proposition}
As a corollary, there is a.s.\ no meromorphic continuation of $\zeta_P$ beyond
the line $\sigma=1/2$. Using the same method of proof it can be shown that for
every different $\tau_1,\tau_2>0$ the random variables $\sqrt{2\sigma-1}\,
\zeta_P(\sigma+i\tau_j)$, $j=1,2$, become asymptotically independent as
$\sigma\downarrow 1/2$. Thus, the function $\zeta_P$ looks like a na\"{\i}ve
white noise near the line $\sigma=1/2$. The intensity of complex zeros of
$\zeta_P$ at $\beta$ can be computed by the formula $g(\beta)=\frac 1 {2\pi}
\Delta \E  \log |\zeta_P(\beta)|$, where $\Delta$ is the Laplace operator;
see~\cite[Section~2.4]{peres_etal_book}. Proposition~\ref{prop:zeta_boundary}
suggests that $g(\beta) \sim \frac 1 \pi \frac1{(2\sigma-1)^2}$ as
$\sigma\downarrow 1/2$. In particular, every point of the line $\sigma=1/2$
should be an accumulation point for the zeros of $\zeta_P$ with probability $1$.

Let us look locally at the zeros of $\ZZZ_N$ near some
$\beta_0=\sigma_0+i\tau_0$ on one of the boundaries $\bar B_1\cap \bar B_3$ or
$\bar B_1\cap \bar B_2$. We will show that in both cases the zeros form
approximately an arithmetic sequence. The structure of the measures $\Xi_{13}$
and $\Xi_{12}$ in Theorem~\ref{theo:zeros_in_B3} suggests that the distances
between the consequent zeros should behave like $\frac{2\pi}{n}$ in the first
case and like $\frac{\sqrt 2\pi}{|\tau_0|n}$ in the second case. The next
theorems show that this is indeed true. First, we analyze the boundary $\bar
B_1\cap \bar B_3$.
\begin{theorem}\label{theo:boundary_zeros1}
Let $\beta_0=\sigma_0 + i\tau_0$ be such that $\sigma_0^2+\tau_0^2=1$ and
$\sigma_0^2<1/2$.  There exist a complex-valued random variable $\xi$ and a
bounded real sequence $\delta_N$ such that for every continuous function
$f \colon \C\to\R$ with compact support,
$$
\sum_{\beta\in \C: \ZZZ_N(\beta)=0}
f\left(n \left(\frac{\beta}{\beta_0} - 1 \right)- i \delta_N\right)
\todistr
\sum_{k\in\Z} f(2\pi i k + \xi).
$$
\end{theorem}
\begin{remark}
In other words, the zeros of $\ZZZ_N$ near $\beta_0$ are given by the formula
$$
\beta=\beta_0\left(1+\frac{2\pi i k + \xi +i\delta_N}{n}\right)+o\left(\frac 1n \right), \;\;\; k\in\Z.
$$
As we will see in the proof, the random variable $\Re \xi$ takes  negative
values with positive probability. It follows that the probability that $\ZZZ_N$
has a zero in $B_1$ does not go to $0$ as $N\to\infty$.
\end{remark}
The boundary $\bar B_1\cap \bar B_2$ consists of $4$ line segments. By symmetry \eqref{eq:symmetries},
it suffices to consider one of them.
\begin{theorem}\label{theo:boundary_zeros2}
Let $\beta_0=\sigma_0+i\tau_0$ be such that $\sigma_0>1/\sqrt 2$, $\tau_0>0$ and
$\sigma_0+\tau_0=\sqrt 2$.  There exist a complex-valued random variable $\eta$
and a complex sequence $d_N=O(\log n)$ such that for every continuous function
$f\colon \C\to\R$ with compact support,
$$
\sum_{\beta\in \C: \ZZZ_N(\beta)=0} f\left( \ee^{\frac {2\pi i}3} n (\beta-\beta_0) - d_N\right)
\todistr
\sum_{k\in\Z} f\left(\frac{2\pi i k + \eta}{\sqrt 2 \tau_0} \right).
$$
\end{theorem}
\begin{remark}
In other words, the zeros of $\ZZZ_N$ near $\beta_0$ are given by the formula
$$
\beta=\beta_0+\ee^{-\frac{2\pi i}{3}}\frac 1n \left(\frac{2\pi i k}{\sqrt 2 \tau_0}+d_N\right)+o\left(\frac 1n \right), \;\;\; k\in\Z.
$$
\end{remark}

\subsection{Results on fluctuations}\label{subsec:fluct} We state our results on
fluctuations for a generalization of~\eqref{eq:def_ZZZ_N} which we call complex
random energy model. This model involves complex phases and allows for a
dependence between the energies and the phases. Let $(X,Y), (X_1,Y_1), \ldots$
be i.i.d.\ zero-mean bivariate Gaussian random vectors with
$$
\var X_k=\var Y_k=1, \;\;\; \corr (X_k,Y_k)=\rho.
$$
Here, $-1\leq \rho\leq 1$ is fixed.  
Recall \eqref{eq:def_ZZZ_N} and consider the following partition function:
\begin{align}
\label{eq:partition_function}
\ZZZ_N(\beta)
=
\sum_{k=1}^{N}
\ee^{\sqrt{n} (\sigma X_k + i \tau Y_k)},
\qquad
\beta=(\sigma,\tau)\in \R^2.
\end{align}
For $\tau=0$, this is the REM of~\citet{Derrida1981} at real inverse temperature
$\sigma$. For $\rho=1$, we obtain the REM at the complex inverse temperature
$\beta= \sigma + i \tau$ considered above; see~\eqref{eq:def_ZZZ_N}. For
$\rho=0$, the model is a REM with independent complex phases considered
in~\cite{DerridaEvansSpeer1993}. Note also that the substitutions
$(\beta,\rho)\mapsto (-\beta, \rho)$ and $(\beta,\rho)\mapsto (\bar \beta,
-\rho)$ leave the distribution of $\ZZZ_N(\beta)$ unchanged.

Recall \eqref{eq:limiting_log_partition_function_real}. Define the log-partition function as
\begin{align}
\label{eq:log_partition_function}
p_N(\beta)
=
\frac{1}{n}
\log | \ZZZ_N(\beta)|,
\qquad
\beta=(\sigma,\tau)\in \R^2.
\end{align}
\begin{theorem}\label{theo:log_scale}
For every $\beta\in\R^2$, the limit
\begin{equation}
\label{eq:log-partition-function-limit}
p(\beta)
:=
\lim_{N \to \infty}
p_N(\beta)
\end{equation}
exists in probability and in $L^q$, $q\geq 1$, and is explicitly given as
\begin{equation}\label{eq:limiting_log_partition_function_1}
p(\beta)=
\begin{cases}
1+\frac{1}{2}(\sigma^2 - \tau^2),
&\beta \in \overline{B}_{1},\\
\sqrt{2} |\sigma|,
&\beta \in \overline{B}_2,\\
\frac 12 +\sigma^2,
&\beta \in \overline{B}_3.
\end{cases}
\end{equation}
\end{theorem}
Note that the limit in~\eqref{eq:limiting_log_partition_function_1} does not depend on $\rho$. However, we will see below that the fluctuations of $\ZZZ_N(\beta)$ do depend on $\rho$. The next theorem shows that $\ZZZ_N(\beta)$ satisfies the central limit theorem in the domain $\sigma^2<1/2$.
\begin{theorem}\label{theo:fluct_small_sigma} 
If $\sigma^2<1/2$ and $\tau\neq 0$, then
\begin{equation}\label{eq:fluct_small_sigma}
\frac{\ZZZ_N(\beta)-N^{1+\frac 1 2 (\sigma^2-\tau^2)+ i\sigma\tau\rho}}{N^{\frac 12+ \sigma^2}}\todistr N_{\C}(0,1).
\end{equation}
\end{theorem}
\begin{remark}
If $\sigma^2<1/2$ and $\tau=0$, then the limiting distribution is real normal, as was shown in~\cite{bovier_kurkova_loewe}.
\end{remark}
\begin{remark}\label{rem:clt_simplification_B3}
If in addition to $\sigma^2<1/2$ we have $\sigma^2+\tau^2>1$, then $N^{1 + \frac
1 2 (\sigma^2-\tau^2)}=o(N^{\frac 12 + \sigma^2})$ and, hence, the theorem
simplifies to
\begin{equation}\label{eq:clt_simplification_B3}
\frac{\ZZZ_N(\beta)}{N^{\frac 12 +\sigma^2}}\todistr N_{\C}(0,1).
\end{equation}
Eq.~\eqref{eq:clt_simplification_B3} explains the difference between phases
$B_1$ and $B_3$: in phase $B_1$ the expectation of $\ZZZ_N(\beta)$ is of larger
order than the mean square deviation, in phase $B_3$ otherwise.
\end{remark}

In the boundary case $\sigma^2=1/2$, the limiting distribution is normal, but it has  truncated variance.
\begin{theorem}\label{theo:fluct_crit_sigma}  
If $\sigma^2=1/2$ and $\tau\neq 0$, then
$$
\frac{\ZZZ_N(\beta)-N^{1+\frac 1 2 (\frac 12-\tau^2)+ i\sigma \tau\rho}}{N}\todistr N_{\C}(0,1/2).
$$
\end{theorem}

Next, we describe the fluctuations of $\ZZZ_N(\beta)$ in the domain
$\sigma^2>1/2$. Due to \eqref{eq:symmetries}, it is not a restriction of
generality to assume that $\sigma>0$. Let $b_N$ be a sequence such that $\sqrt
{2\pi}b_N \ee^{b_N^2/2} \sim N$ as $N\to\infty$. We can take
\begin{equation}\label{eq:def_wN}
b_N=\sqrt {2n}- \frac {\log (4\pi n)}{2\sqrt {2 n}}. 
\end{equation}
\begin{theorem}\label{theo:fluct_large_sigma}   
Let $\sigma>1/\sqrt 2$, $\tau\neq 0$, and $|\rho|<1$. Then,
\begin{equation}\label{eq:theo:fluct_large_sigma}
\frac{\ZZZ_N(\beta)-N\E[\ee^{\sqrt n(\sigma X + i\tau Y)}\ind_{X<b_N}]}{\ee^{\sigma \sqrt{n}b_N}}
\todistr
S_{\sqrt {2}/\sigma},
\end{equation}
where $S_{\alpha}$ denotes a complex isotropic $\alpha$-stable random
variable with a characteristic function of the form $\E [\ee^{i \Re
(S_{\alpha}\overline z)}]=\ee^{-\mathrm{const} \cdot |z|^{\alpha}}$, $z\in\C$.
\end{theorem}
\begin{remark}\label{rem:fluct_large_sigma_rho1}
If $\sigma>1/\sqrt 2$  and $\tau=0$, then  the limiting distribution is real totally skewed $\alpha$-stable; see~\cite{bovier_kurkova_loewe}. If $\sigma>1/\sqrt 2$  and $\rho=1$ (resp., $\rho=-1$), then it follows from Theorem~\ref{theo:weak_conv_large_beta} below that
\begin{equation}\label{eq:theo:fluct_large_sigma_rho_1}
\frac{\ZZZ_N(\beta)-N\E[\ee^{\beta \sqrt n X}\ind_{X<b_N}]}{\ee^{\beta \sqrt{n} b_N}}
\todistr
\tilde \zeta_{P} \left(\frac{\beta}{\sqrt 2}\right)
\;\;\;
\left(\text{resp., } \tilde \zeta_{P} \left(\frac{\bar \beta}{\sqrt 2}\right) \right).
\end{equation}

\end{remark}
\begin{remark}\label{rem:truncated_moment}
We will compute asymptotically the truncated expectation on the left-hand side of~\eqref{eq:theo:fluct_large_sigma} in Section~\ref{sec:proofs_fluct_sub} below.  We will obtain that under the assumptions of Theorem~\ref{theo:fluct_large_sigma},
\begin{align}
\frac{\ZZZ_N(\beta)}{\ee^{\sigma \sqrt{n}b_N}}
\todistr
S_{\sqrt {2}/\sigma}, &\;\;\; \text{ if } \sigma+|\tau|>\sqrt 2,\label{eq:trunc_expectation1}\\
\frac{\ZZZ_N(\beta) - N^{1+\frac 12 (\sigma^2-\tau^2)+ i\sigma\tau\rho}}{\ee^{\sigma \sqrt{n}b_N}}
\todistr
S_{\sqrt {2}/\sigma}, &\;\;\; \text{ if } \sigma+|\tau|\leq \sqrt 2. \label{eq:trunc_expectation2}
\end{align}
Similarly, if $\sigma>1/\sqrt 2$, but $\rho=1$, then we have
\begin{align}
\frac{\ZZZ_N(\beta)}{\ee^{\beta \sqrt{n}b_N}}
\todistr
\zeta_{P} \left(\frac{\beta}{\sqrt 2}\right), &\;\;\; \text{ if } \sigma+|\tau|>\sqrt 2,\label{eq:trunc_expectation1a}\\
\frac{\ZZZ_N(\beta) - N^{1+\frac 12 (\sigma^2-\tau^2)+ i\sigma\tau}}{\ee^{\beta \sqrt{n}b_N}}
\todistr
\zeta_{P} \left(\frac{\beta}{\sqrt 2}\right), &\;\;\; \text{ if } \sigma+|\tau|\leq \sqrt 2,\; \sigma\neq \sqrt 2. \label{eq:trunc_expectation2a}
\end{align}
For $\rho=-1$, we have to replace $\beta$ by $\bar \beta$.
\end{remark}

\subsection{Discussion, extensions and open questions}\label{sec:extensions}
The results on fluctuations are closely related, at least on the heuristic level, to
the results on the zeros of $\ZZZ_N$.  In
Section~\ref{subsec:fluct}, we claimed that regardless of the value of $\beta\neq 0$ we can
find normalizing constants $m_N(\beta)\in\C$, $v_N(\beta)>0$ such that
$$
\frac{\ZZZ_N(\beta)-m_N(\beta)}{v_N(\beta)}\todistr Z(\beta)
$$
for some non-degenerate random variable $Z(\beta)$. It turns out that in phase
$B_1$ the sequence $m_N(\beta)$ is of larger order than $v_N(\beta)$, which
suggests that there should be no zeros in this phase. In phases $B_2$ and $B_3$,
 the sequence $m_N(\beta)$ is of smaller order than $v_N(\beta)$, which does not
rule out the possibility of zeros in these phases. One way to guess the density
of zeros in phases $B_2$ and $B_3$ is to look more closely at the correlations
of the process $\ZZZ_N$. In phase $B_3$, it can be seen from
Theorem~\ref{theo:weak_conv_small_beta} below that $\ZZZ_N(\beta_1)$ and
$\ZZZ_N(\beta_2)$ become asymptotically decorrelated if the distance between
$\beta_1$ and $\beta_2$ is of order larger than $1/\sqrt n$. This suggests that
the distances between the close zeros in phase $B_3$ should be of order $1/\sqrt
n$ and hence, the density of zeros  should be of order $n$. Similarly, in phase
$B_2$ the variables $\ZZZ_N(\beta_1)$ and $\ZZZ_N(\beta_2)$ remain non-trivially
correlated at distances of order  $1$ by Theorem~\ref{theo:weak_conv_large_beta}
below, which suggests that the density of zeros in this phase should be of order
$1$.

An additional motivation for studying $\ZZZ_N$ comes from its connection to the
empirical characteristic function. Given an i.i.d.\ standard normal sample
$X_1,\ldots,X_N$, the empirical characteristic function is defined by
$c_N(\beta)=\sum_{k=1}^N \ee^{i \beta X_k}$. We have $\ZZZ_N(\beta)=c_N(-i\sqrt n
\beta)$. The limit behavior of the stochastic process $\{c_N(\beta) \colon
\beta\in\R\}$ without rescaling $\beta$ by the factor $\sqrt n$ has been much
studied; see, e.g.,~\cite{feuerverger_mureika,csoergoe}. There has been
also interest in the behavior of $R_N=\inf \{\beta>0 \colon \Re c_N(\beta)=0\}$, the
first real zero of $\Re c_N$; see~\cite{heathcote_huesler,huesler}. In
particular, it has been shown in~\cite[Corollary~4.5]{heathcote_huesler} that, for all $t\in \R$,
$$
\lim_{N\to\infty}\P[R_N^2-n < 2t]=\Phi(-\sqrt 2 e^{-t}).
$$
Hence, the first real zero of $\Re \ZZZ_N(\beta)$ restricted to $\beta\in i \R$ is located near $i$ with high probability. This is exactly the point where the imaginary axis meets the set $B_3$.

It is possible to extend or strengthen our results in several directions. The
statements of Theorem~\ref{theo:zeros_in_B3} and Theorem~\ref{theo:log_scale}
should hold almost surely, although it seems difficult to prove this.
Several authors considered models  involving sums of random exponentials
generalizing the REM; see~\cite{ben_arous_bogachev_molchanov, bogachev_weibull,
janssen, cranston_molchanov}. They analyze the case of real $\beta$ only. We
believe that our results (both on zeros and on fluctuations) can be extended, with
appropriate modifications, to these models.

\section{Proofs of the results on fluctuations} \label{sec:proofs_fluct}

\subsection{Truncated exponential moments} We will often need estimates for the
truncated exponential moments of the normal distribution. In the next lemmas, we
denote by $X$ a real standard normal random variable. Let $\Phi(z)=\frac{1}{\sqrt{2\pi}}\int_{-\infty}^z \ee^{-\frac{x^2}2}\dd x$ be the distribution function of $X$. It is well-known that
\begin{equation}\label{eq:Phi_asympt}
\Phi(z)
=
-\frac {1+o(1)} {\sqrt {2\pi} z} \ee^{-\frac {z^2}2},
\quad
z\to -\infty.
\end{equation}
The normal distribution function $\Phi$ can be extended as an analytic function to the entire complex plane. We need an extension of~\eqref{eq:Phi_asympt} to the complex case. 
\begin{lemma}\label{lem:Phi_asympt_complex}
Fix some $\eps>0$. The following holds as $|z|\to\infty$, $z\in\C$:
\begin{equation}\label{eq:Phi_asympt_complex}
\Phi(z)=
\begin{cases}
-\frac{1+o(1)}{\sqrt{2\pi}z} \ee^{-\frac{z^2}2},
& \text{ if } |\arg z|>\frac{\pi}{4}+\eps,\\
1-\frac{1+o(1)}{\sqrt{2\pi}z} \ee^{-\frac{z^2}2},
& \text{ if } |\arg z|<\frac{3\pi}{4}-\eps.
\end{cases}
\end{equation}
In particular, $\Phi(z)\to 1$ if $|z|\to\infty$ and $|\arg z|<\frac{\pi}{4}-\eps$.
\end{lemma}
\begin{remark}
We take the principal value of the argument, ranging in $(-\pi,\pi]$ and having
a jump on the negative half-axis. In the domain $\frac{\pi}{4}+\eps <|\arg
z|<\frac{3\pi}{4}-\eps$ both asymptotics in~\eqref{eq:Phi_asympt_complex} can be
applied. To see that they give the same result, note that
$|\frac{1}{z}\ee^{-\frac {z^2}2}|\to\infty$ there.
\end{remark}
\begin{proof}[Proof of Lemma~\ref{lem:Phi_asympt_complex}]
For the first case of~\eqref{eq:Phi_asympt_complex}, see~\citep[Eq.~7.1.23 on
p.~298]{abramowitz_stegun}. The second case
of~\eqref{eq:Phi_asympt_complex} follows from the identity $\Phi(z)=1-\Phi(-z)$.
\end{proof}

In the next lemmas, we record several simple facts on the truncated exponential moments which we will often use  later. Note that $\E [\ee^{wX}]=\ee^{\frac{w^2}2}$ for all $w\in\C$.
\begin{lemma}\label{lem:trunc_moment_Phi}
Let $w\in \C$, $a\in\R$. Then, $\E [\ee^{w X}\ind_{X < a}]=\ee^{\frac{w^2}2} \Phi(a-w)$.
\end{lemma}
\begin{proof}
For $w\in \R$, we have
$$
\E [\ee^{w X}\ind_{X < a}]
=
\frac{1}{\sqrt {2\pi}} \int_{-\infty}^{a} \ee^{wz-\frac{z^2}2}\dd z
=
\frac{1}{\sqrt {2\pi}} \ee^{\frac{w^2}2} \int_{-\infty}^{a} \ee^{-\frac{(z-w)^2}2}\dd z
=
\ee^{\frac{w^2}2} \Phi(a-w).
$$
For $w\in\C$, this holds by analytic continuation.
\end{proof}

\begin{lemma}\label{lem:saddle_weak}
Let $w,a\in\R$. The following estimates hold.
\begin{enumerate}
\item \label{lem:saddle_weak1} If $w>a$, then $\E [\ee^{w X}\ind_{X < a}]<\ee^{aw-\frac{a^2}{2}}$.
\item \label{lem:saddle_weak2} If $w<a$, then $\E [\ee^{w X}\ind_{X > a}]<\ee^{aw-\frac{a^2}{2}}$.
\end{enumerate}
\end{lemma}
\begin{proof}
Consider the case $w>a$. By Lemma~\ref{lem:trunc_moment_Phi}, $\E [\ee^{w X}\ind_{X < a}]=\ee^{\frac{w^2}2} \Phi(a-w)$.
Using the inequality $\Phi(z)<\ee^{-\frac{z^2}2}$ valid for $z\leq 0$, we obtain the statement of case~(1). Case~(2) can be reduced to case~(1) by the substitution $(X,w,a)\mapsto(-X,-w,-a)$.
\end{proof}
\begin{lemma}\label{lem:saddle_point}
Let $F(n)=\E [\ee^{w \sqrt n  X}\ind_{X<\sqrt n a(n)}]$, where  $w=u+iv\in\C$, $n>0$, and $a(n)$ is a real-valued function with $\lim_{n\to \infty}a(n)=a$.
The following hold, as $n\to \infty$:
\begin{equation}\label{eq:lemma_saddle}
F(n)
\sim
\begin{cases}
\frac{1}{\sqrt{2\pi n}(w-a)}\,
\ee^{n (a(n) w -\frac 12 a^2(n))},
& \text{ if } u+|v|>a,\\
\ee^{\frac 12  w^2 n},
& \text{ if } u+|v|<a.
\end{cases}
\end{equation}
If $w\in\R$ and $a(n)=w+\frac{c}{\sqrt n}+o(\frac 1 {\sqrt n})$, for some $c\in\R$, then
\begin{equation}\label{eq:lemma_saddle_crit_real}
F(n)
\sim
\Phi(c)\, \ee^{\frac 12 w^2 n}, \;\;\;n\to \infty.
\end{equation}
\end{lemma}
\begin{remark}\label{rem:saddle}
The second line in~\eqref{eq:lemma_saddle} can be generalized to the following formula valid in the case $u-|v|<a$:
\begin{equation}\label{eq:lem_saddle_complex_two_terms}
F(n)=\ee^{\frac 12 w^2 n}+\frac{1+o(1)}{\sqrt{2\pi n}(w-a)}\,
\ee^{n (a(n) w -\frac 12 a^2(n))}, \;\;\;n\to \infty.
\end{equation}
\end{remark}
\begin{proof}[Proof of Lemma~\ref{lem:saddle_point}]
Let  $z(n)=\sqrt n a(n)- w \sqrt n$. By Lemma~\ref{lem:trunc_moment_Phi}, we have
\begin{equation}\label{eq:F_lambda_Phi}
F(n)
=
\ee^{\frac12 w^2 n }  \Phi(\sqrt n a(n)-w\sqrt n)
=
\ee^{\frac12 w^2 n }  \Phi(z(n)).
\end{equation}
Note that $z(n)\sim (a-u-iv)\sqrt n$, as $n\to \infty$.

\noindent \textit{Case 1.} If $u+|v|>a$, then $|\arg z(n)|>\frac {\pi}{4}+\eps$,
for some $\eps>0$, and all sufficiently large $n$. Applying the first line
of~\eqref{eq:Phi_asympt_complex}, we arrive at the first line
of~\eqref{eq:lemma_saddle}.

\noindent\textit{Case 2.} If $u-|v|<a$, then $|\arg z(n)|<\frac {3\pi}{4}-\eps$,
for some $\eps>0$, and all sufficiently large $n$. Applying the second line
of~\eqref{eq:Phi_asympt_complex}, we get~\eqref{eq:lem_saddle_complex_two_terms}.

\noindent\textit{Case 2a.} If even the stronger condition $u+|v|<a$ holds, then
$\frac 12 (u^2-v^2)>au-\frac 12 a^2$ and hence, the first term
in~\eqref{eq:lem_saddle_complex_two_terms} asymptotically dominates the second
one. We obtain the second line of~\eqref{eq:lemma_saddle}.

\noindent\textit{Case 3.} If $a=w\in\R$ and $a(n)=w+\frac{c}{\sqrt n}+o(\frac 1
{\sqrt n})$, for some $c\in\R$, then $\lim_{n\to\infty} z(n)=c$ and we arrive
at~\eqref{eq:lemma_saddle_crit_real}.

\end{proof}

\begin{lemma}\label{lem:moments_Wn}
If $(X,Y)$ is a real Gaussian vector with standard marginals and correlation $\rho$, then, for $s,a\in\R$,
$$
\E[\ee^{s(\sigma X + i \tau Y)} \ind_{X<a} ] = \ee^{-s^2\tau^2(1-\rho^2)/2} \E [\ee^{s(\sigma+i\tau\rho)X}\ind_{X<a}].
$$
In particular, $\E[\ee^{s(\sigma X + i \tau Y)}]=\ee^{s^2(\sigma^2-\tau^2+2i\sigma \tau \rho)/2}$.
\end{lemma}
\begin{proof}
We have a distributional equality $(X,Y)\overset{d}{=}(X, \rho X+\sqrt{1-\rho^2} W)$, where $(X,W)$ are independent standard normal real random variables. It follows that
\begin{align*}
\E[\ee^{s(\sigma X + i \tau Y)} \ind_{X<a} ]
&=
\E [\ee^{s(\sigma+i\tau\rho)X+ is\tau \sqrt{1-\rho^2}W} \ind_{X<a} ]\\
&=
\ee^{-s^2\tau^2(1-\rho^2)/2} \E [\ee^{s(\sigma+i\tau\rho)X} \ind_{X<a} ],
\end{align*}
where we have used that $\E [\ee^{tW}]=\ee^{t^2/2}$ and that $W$ and $X$ are independent.
\end{proof}

\subsection{Proof of Theorems~\ref{theo:fluct_small_sigma},
\ref{theo:fluct_crit_sigma}, \ref{theo:fluct_large_sigma}}
\label{sec:proofs_fluct_sub} The main tool to prove the results on the
fluctuations is the  summation theory of triangular arrays of random vectors;
see~\cite{gnedenko_kolmogorov_book} and~\cite{meerschaert_book}. The following
theorem can be found in~\cite[\S25]{gnedenko_kolmogorov_book} in the
one-dimensional setting and in~\cite{rvaceva} or in~\cite[Theorem~3.2.2]{meerschaert_book} in the $d$-dimensional setting. Denote by
$|\cdot|$ the Euclidean norm and by $\langle\cdot,\cdot\rangle$ the Euclidean
scalar product.
\begin{theorem}\label{theo:gnedenko}
For every $N\in\N$, let a series $W_{1,N},\ldots,W_{N,N}$ of independent random
$d$-dimensional vectors be given. 
 Assume that, for some locally finite measure $\nu$ on $\R^d\bsl \{0\}$, and
some positive semidefinite matrix $\Sigma$, the following conditions hold:
\begin{enumerate}
\item \label{cond:gned1} $\lim_{N\to\infty} \sum_{k=1}^N \P[W_{k,N}\in
B]=\nu(B)$, for every Borel set $B\subset \R^d\bsl\{0\}$ such that $\nu(\partial
B)=0$.

\item \label{cond:gned2} The following limits exist:
$$
\Sigma
=
\lim_{\eps\downarrow 0} \limsup_{N\to\infty} \sum_{k=1}^N \Var [W_{k,N} \ind_{|W_{k,N}|<\eps}]
=
\lim_{\eps\downarrow 0} \liminf_{N\to\infty} \sum_{k=1}^N \Var [W_{k,N} \ind_{|W_{k,N}|<\eps}].
$$
\end{enumerate}
Then, the random vector $S_N:=\sum_{k=1}^N (W_{k,N}-\E [W_{k,N}
\ind_{|W_{k,N}|<R}])$ converges, as $N\to\infty$, to an infinitely divisible
random vector $S$ whose characteristic function is given by the
L\'evy--Khintchine representation
$$
\log \E [\ee^{i \langle t, S\rangle}]=-\frac 12 \langle t, \Sigma t\rangle + \int_{\R^d} (\ee^{i \langle t, s \rangle}-1- i \langle t, s \rangle \ind_{|s|<R}) \nu(\dd s),
\qquad
t\in\R^d.
$$
Here, $R>0$ is any number such that $\nu$ does not charge the set $\{s\in\R^d \colon |s|=R\}$.
\end{theorem}

\begin{proof}[Proof of Theorem~\ref{theo:fluct_small_sigma}]
For $k=1,\ldots,N$, define
$$
W_{k,N} = N^{-\frac 12-\sigma ^2} \ee^{\sqrt {n} (\sigma X_k+i \tau Y_k)}.
$$
Let $W_N$ be a random variable having the same law as the $W_{k,N}$'s.
 Note that $N\E[W_{N}] = N^{(1-\sigma ^2-\tau^2+2i\sigma \tau\rho)/2}$ by
Lemma~\ref{lem:moments_Wn}. To prove the theorem, we need to show that
$$
\sum_{k=1}^N (W_{k,N}-\E W_{k,N})
\todistr
N_{\C}(0,1).
$$
The proof is based on the two-dimensional Lindeberg central limit
theorem. We consider $W_{k,N}$ as an $\R^2$-valued random vector $(\Re W_{k,N},
\Im W_{k,N})$. Let $\Sigma_N$ be the covariance matrix of this vector. First,
we show that
\begin{equation}\label{eq:clt_cov_limit}
\lim_{N\to\infty} N\Sigma_N
=
\left(
\begin{matrix}
1/2 & 0\\
0   & 1/2
\end{matrix}
\right).
\end{equation}
We have
\begin{equation}\label{eq:proof_clt1}
N\E[(\Re W_N)^2+(\Im W_N)^2]=N\E [|W_N|^2]=N^{-2\sigma ^2} \E [\ee^{2\sigma \sqrt n  X}]=1.
\end{equation}
Also, we have $N \E [W_N^2]=N^{-2\tau^2 +4i \sigma \tau\rho}$ by
Lemma~\ref{lem:moments_Wn}. Since we assume that $\tau\neq 0$, this implies that
$\lim_{N\to\infty}N \E [W_N^2]=0$. By taking real and imaginary parts, we obtain
that
\begin{equation}\label{eq:proof_clt2}
\lim_{N\to\infty} N\E[(\Re W_N)^2-(\Im W_N)^2]=
\lim_{N\to\infty} N\E [(\Re W_N) (\Im W_N)]=0.
\end{equation}
Combining~\eqref{eq:proof_clt1} and~\eqref{eq:proof_clt2}, we get
\begin{equation}\label{eq:proof_clt4}
\lim_{N\to\infty} N\E [(\Re W_N)^2]=\lim_{N\to\infty} N\E [(\Im W_N)^2]=1/2.
\end{equation}
Also, by Lemma~\ref{lem:moments_Wn}, we have
\begin{equation}\label{eq:proof_clt3}
\lim_{N\to\infty} \sqrt{N} \E[W_N]
=
\lim_{N\to\infty} N^{-(\sigma ^2+\tau^2-2i\sigma \tau\rho)/2}=0.
\end{equation}
It follows from~\eqref{eq:proof_clt2}, \eqref{eq:proof_clt4}, \eqref{eq:proof_clt3} that~\eqref{eq:clt_cov_limit} holds.
Fix an arbitrary $\eps>0$. We complete the proof of the theorem by verifying the Lindeberg condition
\begin{equation}\label{eq:clt_lindeberg}
\lim_{N\to\infty} N \E [|W_N-\E W_N|^2 \ind_{|W_N-\E W_N|>\eps}]=0.
\end{equation}
Assume first that $\sigma\neq 0$, say $\sigma>0$. Write
$
a_N=\sigma + \frac 1 {2\sigma}+ \frac {\log \eps} {\sigma  n}
$.
Then, $\lim_{N\to\infty} a_N=\sigma + \frac 1 {2\sigma}>2\sigma$ by the assumption $\sigma^2<1/2$.
Hence, by Part~\ref{lem:saddle_weak2} of Lemma~\ref{lem:saddle_weak}, we have
\begin{equation}\label{eq:proof_lindeberg}
\lim_{N\to\infty} N \E [|W_N|^2 \ind_{|W_N|>\eps}]
=
\lim_{N\to\infty} \ee^{-2\sigma ^2 n}\E [\ee^{2\sigma \sqrt n  X}\ind_{X> \sqrt n a_N}]
=
0.
\end{equation}
This also trivially holds for $\sigma=0$.
Together with~\eqref{eq:proof_clt3}, \eqref{eq:proof_lindeberg} implies~\eqref{eq:clt_lindeberg}.
\end{proof}

\begin{proof}[Proof of Theorem~\ref{theo:fluct_crit_sigma}]
Without loss of generality, let $\sigma=1/\sqrt 2$. For $k=1,\ldots,N$, define
$$
W_{k,N}=N^{-1} \ee^{\sqrt n (\sigma X_k +i \tau Y_k)}.
$$
Let $W_N$ be a random variable with the same distribution as the $W_{k,N}$'s.
To prove the theorem, we need to verify that
$$
\sum_{k=1}^N (W_{k,N}-\E W_{k,N})
\todistr
N_{\C}(0,1/2).
$$
As we will see in equation~\eqref{eq:proof_crit_1} below, the Lindeberg
condition~\eqref{eq:clt_lindeberg} is not satisfied. We are going to apply
Theorem~\ref{theo:gnedenko} instead. Fix $\eps>0$ and let $a_N=\sqrt
{2}+\frac{\sqrt {2}\log \eps}{n}$. By Lemma~\ref{lem:saddle_point},
equation~\eqref{eq:lemma_saddle_crit_real} with $c=0$, we have
\begin{equation}\label{eq:proof_crit_1}
\lim_{N\to\infty} N\E[|W_{N}|^2 \ind_{|W_N|<\eps}]
=
\lim_{N\to\infty} N^{-1}\E[ \ee^{\sqrt{2n} X} \ind_{ X<\sqrt n a_N }]
=
1/2.
\end{equation}
If $|\rho|\neq 1$, then by Lemma~\ref{lem:moments_Wn} and~\eqref{eq:proof_crit_1},
\begin{equation}\label{eq:proof_crit_2}
\lim_{N\to\infty} N\E[ W_{N}^2 \ind_{|W_N|<\eps} ]
\leq
\lim_{N\to\infty} \ee^{-2n (1-\rho^2) \tau^2} N^{-1} \E [ \ee^{\sqrt {2n} X} \ind_{ X<\sqrt n a_N } ]
=0.
\end{equation}
The result of~\eqref{eq:proof_crit_2} continues to hold for $|\rho|=1$ since in 
this case,  Lemma~\ref{lem:moments_Wn} and Lemma~\ref{lem:saddle_point} (first
part of~\eqref{eq:lemma_saddle}) yield, as $N\to\infty$,
\begin{align*}
N\E[ W_{N}^2 \ind_{|W_N|<\eps} ]
=
N^{-1} \E [ \ee^{2\sqrt n (\sigma+i\tau\rho)X} \ind_{ X<\sqrt n a_N } ]
=
o(N^{-1}\ee^{n(\sqrt 2a_N-\frac12 a_N^2)})
\to 0.
\end{align*}
By Remark~\ref{rem:saddle} we have, as $N\to\infty$, 
$$
\E[\ee^{\sqrt n \sigma X}\ind_{X>\sqrt n a_N}]
\sim \frac{1}{\sqrt{\pi n}} \ee^{n \left(\frac{1}{\sqrt 2}a_N-\frac 12 a_N^2\right)}
\sim \frac{1}{\eps \sqrt{\pi n}}.
$$
It follows that
\begin{equation}\label{eq:proof_crit_3a}
\lim_{N\to\infty} N \E[|W_N|\ind_{|W_N|>\eps}]
=
\lim_{N\to\infty} \E[\ee^{\sqrt n \sigma X}\ind_{X>\sqrt n a_N}]
= 0.
\end{equation}
We consider $W_N$ as an $\R^2$-valued  random vector $(\Re W_N, \Im W_N)$. It follows from~\eqref{eq:proof_crit_1}, \eqref{eq:proof_crit_2}, \eqref{eq:proof_crit_3a} that the covariance matrix $\Sigma_N:=\Var[W_N \ind_{|W_N|<\eps}]$  satisfies
\begin{equation}\label{eq:clt_cov_limit_sigma_crit}
\lim_{N\to\infty} N\Sigma_N
=
\left(
\begin{matrix}
1/4 & 0\\
0   & 1/4
\end{matrix}
\right).
\end{equation}
It follows from~\eqref{eq:proof_crit_3a} that $\lim_{N\to\infty} N \P[|W_N|>\eps] = 0$. Therefore, the
conditions of Theorem~\ref{theo:gnedenko} are satisfied with $\nu=0$ and
$\Sigma$ given by the right-hand side of~\eqref{eq:clt_cov_limit_sigma_crit}.
Applying Theorem~\ref{theo:gnedenko}, we obtain the required statement.
\end{proof}

\begin{proof}[Proof of Theorem~\ref{theo:fluct_large_sigma}]
Recall that
$\alpha=\sqrt {2}/\sigma\in (0,2)$. For $k=1,\ldots,N$, define random variables
$$
W_{k,N}=\ee^{\sqrt {n} (\sigma X_k+i\tau Y_k-\sigma b_N)}.
$$
Let $W_N$ be a random variable having the same law as the $W_{k,N}$'s. We will
verify the conditions of Theorem~\ref{theo:gnedenko}.  To verify the first
condition, fix $0<r_1<r_2$, $0<\theta_1<\theta_2<2\pi$ and consider the set
$$
B=\{z\in \C \colon r_1<|z|<r_2, \theta_1<\arg z<\theta_2\}.
$$
We will show that
\begin{equation}\label{eq:proof_stable_cond1}
\lim_{N\to\infty} N\P[W_N\in B] =
\left(\frac 1 {r_1^{\alpha}}-\frac 1 {r_2^{\alpha}}\right)
\cdot \frac{\theta_2-\theta_1}{2\pi}.
\end{equation}
Define a set
$$
A_N=\bigcup_{j\in\Z} \left(\frac{2\pi j+\theta_1}{\tau \sqrt n}, \frac{2\pi j+\theta_2}{\tau \sqrt n}\right)\subset \R.
$$
We have
\begin{align*}
\P[W_N\in B]
&=
\P[\ee^{\sigma \sqrt {n}  (X -  b_N)} \in (r_1,r_2),  Y \in A_N]\\
&=
\int_{r_1}^{r_2} \P\left[Y\in A_N \mid \sigma \sqrt {n}  (X - b_N)=\log r\right]  f_N(r) \dd r.
\end{align*}
Here, $f_N(r)$ is the density of the log-normal random variable $\ee^{\sqrt {n} \sigma (X - b_N)}$:
\begin{equation}\label{eq:log_normal_dens_asympt}
f_N(r)
=
\frac{1}{\sqrt{2\pi n}\sigma  r} \exp\left\{-\frac{1}{2}\left(\frac{\log r}{\sigma  \sqrt {n}}+ b_N\right)^2\right\}
\sim
\frac 1N \alpha r^{-(1+\alpha)},
\;\;\;
N\to\infty,
\end{equation}
where the asymptotic equivalence holds uniformly in $r\in [r_1,r_2]$. To
prove~\eqref{eq:log_normal_dens_asympt}, recall that $\sqrt
{2\pi}b_N\ee^{b_N^2/2}\sim N$ and $b_N\sim \sqrt{2n}$. Conditionally on $\sigma \sqrt {n}  (X - b_N)=\log
r$, the random variable $Y$ is normal with mean $\mu_N=\rho (\frac{\log r}{\sigma
\sqrt{n}}+b_N)$ and variance $\sqrt{1-\rho^2}$. The variance is strictly
positive by the assumption $|\rho|\neq 1$. It follows easily that
$$
\lim_{N\to\infty} \P[Y \in A_N  \mid \sigma \sqrt {n}  (X - b_N)=\log r] = \frac{\theta_2-\theta_1}{2\pi}.
$$
Bringing everything together, we arrive at~\eqref{eq:proof_stable_cond1}. So,
the first condition of Theorem~\ref{theo:gnedenko} holds with
$$
\nu(\dd x \dd y)=\frac{\alpha}{2\pi}\cdot \frac{\dd x\dd y}{r^{2+\alpha}}, \;\;\; r=\sqrt{x^2+y^2}.  
$$
To verify the second condition of Theorem~\ref{theo:gnedenko} with $\Sigma=0$, it suffices to show that
\begin{equation}\label{eq:proof_stable_cond2}
\lim_{\eps\downarrow 0} \limsup_{N\to\infty} N\E [|W_N|^2 \ind_{|W_N|\leq \eps}]=0.
\end{equation}
Condition $|W_N|\leq \eps$ is equivalent to $X<a_N$, where $a_N=b_N+\frac{\log
\eps}{\sigma \sqrt n}\sim \sqrt {2n}$.  By Lemma~\ref{lem:saddle_point} (first case of~\eqref{eq:lemma_saddle}) with
$\lambda=\sqrt n$, $w=2\sigma$, we have
$$
\E [\ee^{2\sigma \sqrt {n} X}\ind_{X<a_N}]\sim C n^{-1/2} \ee^{2\sigma \sqrt n a_N- a_N^2/2}\sim CN^{-1}\ee^{2\sigma \sqrt n b_N}\eps^{2-\sqrt 2/\sigma}, \;\;\; N\to\infty,
$$
where we have again used that $\sqrt {2\pi}b_N \ee^{b_N^2/2}\sim N$.
We obtain that
$$
\lim_{N\to\infty} N\E [|W_N|^2 \ind_{|W_N|\leq \eps}]
=
\lim_{N\to\infty} N \ee^{-2\sigma \sqrt {n} b_N} \E [\ee^{2\sigma \sqrt {n} X}\ind_{X<a_N}]
=
C\eps^{2-\sqrt 2/\sigma}.
$$
Recalling that $2>\sqrt 2/\sigma$, we arrive at~\eqref{eq:proof_stable_cond2}.
By Theorem~\ref{theo:gnedenko},
$$
\sum_{k=1}^N (W_{k,N}-\E [W_{N}\ind_{|W_N|<1}])\todistr S_{\alpha},
$$
where the limiting
random vector $S_{\alpha}$ is infinitely divisible with a characteristic function  given by
$$
\psi(z):=\log \E [\ee^{i \langle S_\alpha, z \rangle}] = \frac{\alpha}{2\pi} \int_{\R^2} (\ee^{i \langle u, z \rangle}-1-i \langle u, z \rangle \ind_{|u|<1})
\frac{\dd x \dd y}{|u|^{2+\alpha}},
\;\;\;
z\in\C.
$$
Here, $u=x+iy$ and $\langle u, z \rangle= \Re (u \bar z)$. Clearly, $\psi(z)$
depends on $|z|$ only and satisfies $\psi(\lambda z)=\lambda^{\alpha}\psi(z)$
for every $\lambda>0$. It follows that $\psi(z)=\mathrm{const}\cdot |z|^{\alpha}$.
\end{proof}

\begin{proof}[Proof of Remark~\ref{rem:truncated_moment}]
We prove~\eqref{eq:trunc_expectation1} and~\eqref{eq:trunc_expectation2} first. Let $\sigma>1/\sqrt2$, $\tau\neq 0$, and $|\rho|<1$.  By Lemma~\ref{lem:moments_Wn}, we have
\begin{equation}\label{eq:comp_mN}
m_N
:=
N\E[\ee^{\sqrt {n}(\sigma X + i\tau Y)}\ind_{X<b_N}]
=
N^{1-\frac 12 \tau^2(1-\rho^2)} \E[\ee^{\sqrt {n}(\sigma+i\tau\rho)X}\ind_{X<b_N}].
\end{equation}
Write $w=\sigma+i\tau\rho$. Recall from~\eqref{eq:def_wN} that
\begin{equation}\label{eq:recall_bN}
\sqrt {2\pi}b_N \ee^{b_N^2/2}
\sim N,
\;\;\;
b_N\sim \sqrt{2n}, \;\;\; N\to\infty.
\end{equation}
Applying Lemma~\ref{lem:saddle_point}, we obtain
\begin{equation}\label{eq:comp_mN1}
\E[\ee^{\sqrt {n}(\sigma+i\tau\rho)X}\ind_{X<b_N}]
\sim
\begin{cases}
\frac{1}{(w/\sqrt 2)-1} N^{-1}\ee^{w\sqrt n b_N}, &\sigma+|\tau\rho|>\sqrt 2,\\
 \ee^{\frac 12 w^2n}, & \sigma+|\tau \rho| \leq \sqrt 2.
\end{cases}
\end{equation}
In the case $\sigma+|\tau \rho| = \sqrt 2$, we applied Remark~\ref{rem:saddle}
and noted that the first term in~\eqref{eq:lem_saddle_complex_two_terms}
dominates the second one.

\begin{proof}[Proof of~\eqref{eq:trunc_expectation1}]
Assume that $\sigma+|\tau|>\sqrt 2$.
If even the stronger condition $\sigma+|\tau\rho|>\sqrt 2$ is satisfied, then it follows from~\eqref{eq:comp_mN}  and the first line of~\eqref{eq:comp_mN1} that
\begin{equation}\label{eq:tech_fluct1}
|m_N|\sim C\ee^{\sigma \sqrt n b_N - \frac 12 \tau^2(1-\rho^2)n}=o(\ee^{\sigma \sqrt n b_N }).
\end{equation}
The last step follows from $\tau\neq 0$ and $|\rho|<1$.  If $\sigma+|\tau|>\sqrt 2$ but $\sigma+|\tau\rho|\leq \sqrt 2$, then it follows from~\eqref{eq:comp_mN} and the second line of~\eqref{eq:comp_mN1} that
\begin{equation}\label{eq:tech_fluct2}
|m_N|\sim N^{1+\frac 12 (\sigma^2-\tau^2)}=o(\ee^{\sigma\sqrt n b_N}).
\end{equation}
The last step follows from the inequality $1+\frac 12 (\sigma^2-\tau^2)<\sqrt 2 \sigma$. It follows from~\eqref{eq:tech_fluct1} and~\eqref{eq:tech_fluct2} that we can rewrite Theorem~\ref{theo:fluct_large_sigma} in the form~\eqref{eq:trunc_expectation1}.
\end{proof}

\begin{proof}[Proof of~\eqref{eq:trunc_expectation2}]
Assume that $\sigma+|\tau|\leq \sqrt 2$. Then, $\sigma-|\tau\rho|<\sqrt 2$ and it follows from~\eqref{eq:comp_mN} and Remark~\ref{rem:saddle} that
$$
m_N = N^{1-\frac 12 \tau^2(1-\rho^2)} (\ee^{\frac 12 w^2n}+ O(N^{-1}\ee^{w\sqrt n b_N}))
= N^{1+\frac 12 (\sigma^2-\tau^2)+ i\sigma \tau \rho} + o(\ee^{\sigma \sqrt n b_N}).
$$
The last step follows from $\tau\neq 0$ and $|\rho|<1$. Hence,  we can rewrite Theorem~\ref{theo:fluct_large_sigma} in the form~\eqref{eq:trunc_expectation2}.
\end{proof}

We now proceed to the proof of~\eqref{eq:trunc_expectation1a} and~\eqref{eq:trunc_expectation2a}. Let $\sigma>1/\sqrt 2$ and $\rho=1$. Our starting point is~\eqref{eq:theo:fluct_large_sigma_rho_1}.

\begin{proof}[Proof of~\eqref{eq:trunc_expectation1a}]
Assume that $\sigma+|\tau|>\sqrt2$ and $\sigma\neq \sqrt 2$. By Lemma~\ref{lem:saddle_point}, first line of~\eqref{eq:lemma_saddle}, we have
$$
m_N
:=
N\E[\ee^{\beta \sqrt {n} X}\ind_{X<b_N}]
\sim
\frac{N}{\sqrt {2\pi n} (\beta-\sqrt 2)}\ee^{\beta \sqrt n b_N- \frac 12 b_N^2}
\sim
\frac{\sqrt 2}{\beta-\sqrt 2}\ee^{\beta \sqrt n b_N},
,
$$
where we have used~\eqref{eq:recall_bN}.
Recall that
\begin{equation}\label{eq:zeta_zeta_tilde}
\tilde \zeta_P\left(\frac{\beta}{\sqrt 2}\right) + \frac{\sqrt 2}{\beta-\sqrt 2} = \zeta_P\left(\frac{\beta}{\sqrt 2}\right).
\end{equation}
It follows that we can rewrite~\eqref{eq:theo:fluct_large_sigma_rho_1} in the form~\eqref{eq:trunc_expectation1a}.
\end{proof}

\begin{proof}[Proof of~\eqref{eq:trunc_expectation2a}]
Assume that $\sigma+|\tau|\leq \sqrt2$.
By Remark~\ref{rem:saddle},
$$
m_N
:=
N\E[\ee^{\beta \sqrt {n} X}\ind_{X<b_N}]
=
N^{1+\frac 12\beta^2} + \frac{\sqrt 2+o(1)}{\beta-\sqrt 2}\ee^{\beta \sqrt n b_N}.
$$
It follows that we can rewrite~\eqref{eq:theo:fluct_large_sigma_rho_1} in the form~\eqref{eq:trunc_expectation2a}.
\end{proof}

\end{proof}

\subsection{Proof of Theorem~\ref{theo:log_scale}}
\label{subsec:proof_theo_log_scale} We will deduce the stochastic convergence of
the log-partition function $p_N(\beta)=\frac 1 n\log |\ZZZ_N(\beta)|$   from the
weak convergence of $\ZZZ_N(\beta)$. This will be done via
Lemma~\ref{lem:proof_log_scale} stated below. One may ask whether there exists a
more direct way to prove the convergence of $p_N(\beta)$. A standard method to
handle such questions for real $\beta$ is to use the Gaussian concentration inequality;
see~\cite[Theorem~1.3.4]{talagrand_book2}. To apply it we
need to verify that the function $p_N(\beta)$ is Lipschitz in the variables
$X_1,\ldots,X_N$. This is easy to do in the real setting, but if $\beta$ is
complex, we have $\ZZZ_N(\beta)=0$ for some non-empty set of tuples $X_1,\ldots,X_N$.
Thus, $p_N(\beta)$ is not even finite, so that the Lipschitz property does not hold.
The possibility of having infinite $p_N(\beta)$ is not just a technical
difficulty, especially in view of the fact  that the zeros of $\ZZZ_N$ become
dense in $B_3$ in the large $N$ limit.

\begin{lemma}\label{lem:proof_log_scale}
Let $Z,Z_1,Z_2,\ldots$ be  random variables with values in $\C$ and let $m_N\in\C$, $v_N\in \C\bsl \{0\}$ be sequences of normalizing  constants  such that
\begin{equation}\label{eq:proof_log_scale1}
\frac{Z_N-m_N}{v_N}\todistr Z.
\end{equation}
The following two statements hold:
\begin{enumerate}
\item If $|v_N|=o(|m_N|)$ and $|m_N|\to\infty$ as $N\to\infty$, then $\frac{\log |Z_N|}{\log |m_N|}\toprobab 1$.
\item If $|m_N|=O(|v_N|)$, $|v_N|\to \infty$ as $N\to\infty$ and $Z$ has no atoms, then $\frac{\log |Z_N|}{\log |v_N|}\toprobab 1$.
\end{enumerate}
\end{lemma}
\begin{proof}[Proof of (1)]
Fix $\eps>0$. For sufficiently large $N$, we have $|m_N|>1$ and, hence,
\begin{align*}
\P\left[1-\eps<\frac{\log |Z_N|}{\log |m_N|}<1+\eps\right]
&=
\P[|m_N|^{1-\eps}<|Z_N|<|m_N|^{1+\eps}]\\
&\geq
\P\left[\left|\frac{Z_N-m_N}{v_N}\right|<\frac 12 \frac{|m_N|}{|v_N|}\right].
\end{align*}
The right-hand side converges to $1$ by our assumptions.
\end{proof}
\begin{proof}[Proof of (2)]
Fix $\eps>0$. For sufficiently large $N$,
$$
\P\left[\frac{\log |Z_N|}{\log |v_N|}>1+\eps\right]
=
\P\left[\frac{|Z_N|}{|v_N|}>|v_N|^{\eps}\right]
\leq
\P\left[\left|\frac {Z_N-m_N}{v_N}\right|>\frac 12 |v_N|^{\eps}\right].
$$
The right hand-side converges to $0$ by our assumptions. Consider now
$$
\P\left[\frac{\log |Z_N|}{\log |v_N|}<1-\eps\right]
=
\P\left[\frac{|Z_N|}{|v_N|}<|v_N|^{-\eps}\right]
=
\P\left[ \left|\frac {Z_N-m_N}{v_N}+\frac {m_N}{v_N}\right|<|v_N|^{-\eps}\right].
$$
Assume that there is $\delta>0$ such that the right-hand side is greater than
$\delta$ for infinitely many $N$'s. Recall that $m_N/v_N$ is bounded.  Taking a
subsequence, we may assume that $-m_N/v_N$ converges to some $a\in\C$. Recall
that  $|v_N|\to\infty$. But then, for every $\eta>0$,
$$
\P[|Z-a|<\eta]\geq \limsup_{N\to\infty}\P\left[\left|\frac {Z_{N}-m_{N}}{v_{N}}-a\right|<\frac{\eta}{2}\right]>\delta.
$$
This contradicts the assumption that $Z$ has no atoms.
\end{proof}

\begin{proof}[Proof of Theorem~\ref{theo:log_scale}: Convergence in probability]
Let $p(\beta)$ be defined by~\eqref{eq:limiting_log_partition_function_1}. Note
that $p$ is a continuous function. We are going to  prove that for every
$\beta\in\C$, $\lim_{N\to\infty} p_N(\beta)=p(\beta)$ in probability. We may
assume that $\tau\neq 0$ since otherwise the result is known from
equation~\eqref{eq:limiting_log_partition_function_real}; see~\cite{eisele,olivieri_picco}. It follows from Theorems~\ref{theo:fluct_small_sigma},
\ref{theo:fluct_crit_sigma}, \ref{theo:fluct_large_sigma},  and
Remark~\ref{rem:fluct_large_sigma_rho1} that
condition~\eqref{eq:proof_log_scale1} is satisfied with $Z_N=\ZZZ_N(\beta)$ and
an appropriate choice of normalizing sequences $m_N,v_N$. We will now verify
that  Lemma~\ref{lem:proof_log_scale} is applicable.

\noindent \textit{Case 1.} Let $\sigma^2\leq 1/2$. In this case, $m_N$ and $v_N$
are given by Theorems~\ref{theo:fluct_small_sigma}, \ref{theo:fluct_crit_sigma};
see also Remark~\ref{rem:clt_simplification_B3}. Namely, $ |m_N|=N^{1+\frac 12
(\sigma^2-\tau^2)} $ and $v_N=N^{\frac 12+\sigma^2}. $

\noindent \textit{Case 1a.} If in addition to $\sigma^2\leq 1/2$ we have
$\sigma^2+\tau^2<1$, then $1+\frac 12 (\sigma^2-\tau^2)>\frac 12+\sigma^2$ and
we obtain $|v_N|=o(|m_N|)$.

\noindent \textit{Case 1b.} If in addition to $\sigma^2\leq 1/2$ we have
$\sigma^2+\tau^2\geq 1$, then $1+\frac 12 (\sigma^2-\tau^2)\leq \frac
12+\sigma^2$ and we obtain $|m_N|=O(|v_N|)$.

\noindent \textit{Case 2.} Let $\sigma^2>1/2$ and, without restriction of
generality, $\sigma>1/\sqrt 2$. Then, $m_N$ and $v_N$ are given by
Remark~\ref{rem:truncated_moment}. Namely, $|v_N|=\ee^{\sigma\sqrt n b_N}$ and
the formula for $m_N$ depends on whether $\sigma+|\tau|>\sqrt 2$ or not.

\noindent \textit{Case 2a.} If $\sigma>1/\sqrt2$ and $\sigma+ |\tau|>\sqrt 2$,
then $m_N=0$, see~\eqref{eq:trunc_expectation1} and~\eqref{eq:trunc_expectation1a}. Thus, $|m_N|=o(|v_N|)$ is satisfied.

\noindent \textit{Case 2b.} If $\sigma>1/\sqrt2$ and $\sigma+ |\tau|< \sqrt 2$,
then $|m_N|=N^{1+\frac 12 (\sigma^2-\tau^2)}$; see~\eqref{eq:trunc_expectation2}
and~\eqref{eq:trunc_expectation2a}. From the inequality $1+\frac 12
(\sigma^2-\tau^2)>\sqrt 2 \sigma$ it follows that  $|v_N|=o(|m_N|)$.

\noindent \textit{Case 2c.} If $\sigma>1/\sqrt2$ and $\sigma+ |\tau| = \sqrt 2$,
then $1+\frac 12 (\sigma^2-\tau^2)=\sqrt 2 \sigma$. However, since $\sqrt n
b_N-\sqrt{2}n\to -\infty$ by~\eqref{eq:def_wN}, we still have $|v_N|=o(|m_N|)$.

To summarize, the normalizing constants $m_N$ and $v_N$ satisfy the first
condition of Lemma~\ref{lem:proof_log_scale} if $\beta\in B_1$ or $\beta$
belongs to one of four (open) line segments on the boundary of $B_1$ and $B_2$.
Otherwise, $m_N$ and $v_N$ satisfy the second condition of
Lemma~\ref{lem:proof_log_scale}. Note that we need also to verify that the random
variable $\zeta_P(\beta/\sqrt 2)$ has no atoms if $\sigma>1/\sqrt 2$. This will
be done in Lemma~\ref{lem:zeta_no_atoms}, below. Applying
Lemma~\ref{lem:proof_log_scale}, we obtain that $p_N(\beta)\to p(\beta)$ in
probability.
\end{proof}

\begin{lemma}\label{lem:zeta_no_atoms}
If $\sigma>1/2$, then the random variable $\zeta_P(\beta)$ has no atoms in $\C$.
\end{lemma}
\begin{proof}
For a random variable $Y$ with values in $\C$, let $Q(Y)=\sup_{y\in\C}
\P[Y=y]$ be the weight of the maximal atom of $Y$. Note that $Q$ is a special
case of the concentration function; see~\cite[p.~22]{petrov_book}. For
independent random variables $Y_1$ and $Y_2$, the convolution formula implies
that
\begin{equation}\label{eq:ineq_Q_concentration_func}
Q(Y_1+Y_2)\leq \max(Q(Y_1),Q(Y_2)).
\end{equation}
Also,  $Q(Y+c)=Q(Y)$ for every $c\in\C$. Let $P_1<P_2<\ldots$ be the points of a unit intensity Poisson point process on the positive half-line. For $T>1$, we write
\begin{equation}\label{eq:zeta_P_decomp}
\zeta_P(\beta)=\tilde \zeta_P(\beta; T)+ R(\beta; T),
\end{equation}
where $R(\beta; T)$ is a rest term and $\tilde \zeta_P(\beta;T)$ is defined as in~\eqref{eq:zeta_poi_tilde}, that is
\begin{equation}\label{eq:zeta_poi_tilde_repeat}
\tilde \zeta_P(\beta;T)=\sum_{k=1}^{\infty} \frac 1 {P_k^{\beta}}\ind_{P_k\in[0,T]}- \int_1^Tt^{-\beta}\dd t.
\end{equation}
Note that $\tilde \zeta_P(\beta;T)$ and $R(\beta; T)$ are independent random variables since $R(\beta;T)$ depends only on those points of the Poisson process which are in the interval $(T,\infty)$.

We will show that $Q(\tilde \zeta_P(\beta;T))\leq \ee^{-T}$ for every $T>1$.
By~\eqref{eq:ineq_Q_concentration_func} and~\eqref{eq:zeta_P_decomp} this
implies that $Q(\zeta_P(\beta))=0$, which is the desired result. Consider the
random events $A_m(T)=\{\sum_{k=1}^{\infty} \ind_{P_k\in[0,T]}=m\}$, $m\in\N_0$.
That is, $A_m(T)$ occurs if there are exactly $m$ points of the Poisson point
process in $[0,T]$.  Note that $\P[A_0(T)]=\ee^{-T}$ and $\tilde
\zeta_P(\beta;T)$ is  constant on the event $A_0(T)$. Let $z\in\C$. By the total
probability formula,
\begin{equation}\label{eq:zeta_tilde_eq_z_est}
\P[\tilde \zeta_P(\beta;T)=z]
\leq
\ee^{-T}+\sum_{m=1}^{\infty} \P[\tilde \zeta_P(\beta;T)=z | A_m].
\end{equation}
Conditionally on $A_m$, where $m\in\N$, the points $P_1,\ldots,P_m$ have the
same distribution as the increasingly reordered independent random variables
$U_1,\ldots,U_m$ distributed uniformly on $[0,T]$. It follows that, for every
$m\in\N$,
\begin{equation}\label{eq:zeta_tilde_eq_z_est2}
\P[\tilde \zeta_P(\beta;T)=z | A_m]
\leq
Q\left(\sum_{k=1}^m U_k^{-\beta}\right)
\leq
Q( U_1^{-\beta})
=0,
\end{equation}
where the last line follows from the fact that the random variable
$U_1^{-\beta}$ has no atoms. It follows from~\eqref{eq:zeta_tilde_eq_z_est}
and~\eqref{eq:zeta_tilde_eq_z_est2} that $\P[\tilde \zeta_P(\beta;T)=z]\leq
\ee^{-T}$ for every $z\in\C$, $T>1$. This implies that $Q(\tilde
\zeta_P(\beta;T))\leq \ee^{-T}$ and completes the proof.
\end{proof}

\begin{proof}[Proof of Theorem~\ref{theo:log_scale}: Convergence in $L^q$]
We are going to show that $p_N(\beta)\to p(\beta)$ in $L^q$, where $q\geq 1$ is fixed.
From the fact that $p_N(\beta)\to p(\beta)$ in probability, and since $p(\beta)>0$, for every
$\beta\in\C$, we conclude that, for every $C>p(\beta)$,
$$
\lim_{N\to\infty} p_N(\beta)\ind_{0 \leq p_N(\beta) \leq C+1}=p(\beta) \text { in } L^q.
$$
For every $u\in\R$, we have
$$
\P[p_N(\beta)\geq u]
\leq \ee^{-n u} \E |\ZZZ_N(\beta)|
\leq \ee^{-n u} N \E[\ee^{\sigma \sqrt n X}]
=\ee^{n(C-u)},
$$
where $C=1+\sigma^2/2$. From this, we conclude that
\begin{align*}
\E \left[|p_N(\beta)|^q \ind_{p_N(\beta)>C+1}\right]
&=
\sum_{k=1}^{\infty} \E[|p_N(\beta)|^q\ind_{C+k < p_N(\beta) \leq C+k+1}]\\
&\leq
\sum_{k=1}^{\infty} \ee^{-nk}(C+k+1)^q,
\end{align*}
which converges to $0$, as $N\to\infty$.
To complete the proof, we need to show that
\begin{equation}\label{eq:Lq_bound_small_pN}
\lim_{N\to\infty} \E \left[|p_N(\beta)|^q \ind_{p_N(\beta)<0}\right]=0.
\end{equation}
The problem is to bound the probability of small values of $\ZZZ_N(\beta)$,
where the logarithm has a singularity and $|p_N(\beta)|$ becomes large. This is
non-trivial because of the presence of complex amplitudes in the definition of
$\ZZZ_N(\beta)$; see~\eqref{eq:partition_function}. We have to show that there
is not much cancellation among the terms in~\eqref{eq:partition_function}.  Fix a
small $\eps>0$. Clearly,
\begin{equation}\label{eq:Lq_bound_small_pN_step1}
\E \left[|p_N(\beta)|^q \ind_{-\eps \sigma^2 \leq p_N(\beta)\leq 0}\right]\leq (\eps \sigma^2)^q.
\end{equation}
To prove~\eqref{eq:Lq_bound_small_pN}, we would like to estimate from above the
probability $\P[|\ZZZ_N(\beta)|\leq r]$ for $0<r<\ee^{-\eps \sigma^2 n}$. Recall
that $\ZZZ_N(\beta)$ is a sum of $N$ independent copies of the random variable
$\ee^{\sqrt n (\sigma X+ i\tau Y)}$. Unfortunately, the distribution of the
latter random variable does not possess nice regularity properties. For example,
in the most interesting case $\rho=1$ it has no density.  This is why we need a
smoothing argument. Denote by $B_r(t)$ the disc of radius $r$ centered at $t\in
\C$.  Fix a large $A>1$. We will show that uniformly in $t\in\C$, $1/A<|\beta|
<A$, $n> (2A)^2$, and $0<r<\ee^{-\eps \sigma^2 n}$,
\begin{equation}\label{eq:bound_density_small_r}
\P[\ee^{\sqrt n (\sigma X+ i\tau Y)} \in  B_r(t)] < C r^{\frac{\eps}{20}}.
\end{equation}
This inequality is stated in a form which will be needed later in the proof of Theorem~\ref{theo:zeros_in_B3}.

Let $|t| \geq \sqrt r$ and $\tau \geq 1/(2A)$. The argument $\arg t$ of a
complex number $t$ is considered to have values in the circle $\T=\R/2\pi \Z$. Let
$P:\R\to\T$ be the canonical projection. Denote by $I_r(t)$ the sector
$\{z\in \C \colon |\arg z-\arg t|\leq 2\sqrt r\}$, where we take the geodesic
distance between the arguments. A simple geometric argument shows that the disc $B_r(t)$ is
contained in the sector $I_r(t)$. The density of the random variable $P(\tau
\sqrt n Y)$ at $\theta \in \T$ is given by
$$
\P[P(\tau \sqrt n Y)\in \dd\theta] = \frac{1}{\sqrt {2\pi n} \tau}\sum_{k\in\Z} \ee^{-(\theta+2\pi k)^2/(2\tau^2 n)}\dd\theta.
$$
 By considering the right-hand side as a Riemann sum and recalling that $\tau\geq 1/(2A)$, we
see that the density converges to $1/(2\pi)$ uniformly in $\theta\in\T$ as
$N\to\infty$. We have
$$
\P[\ee^{\sqrt n (\sigma X+ i\tau Y)}\in B_r(t)]
\leq
\P[\ee^{\sqrt n (\sigma X+ i\tau Y)}\in I_r(t)]
<
C\sqrt r,
$$
which implies~\eqref{eq:bound_density_small_r}.

Let now $|t|<\sqrt r$. Then, recalling that $\log r<-\eps \sigma^2 n$, we obtain
$$
\P[ \ee^{\sqrt n (\sigma X+ i\tau Y)} \in B_r(t)]
\leq
\P[ \ee^{\sigma \sqrt n X} < r^{1/4}]
=
\P\left[X < \frac {\log r}{ 4 \sigma \sqrt n}\right]
<
\ee^{-\frac{(\log r)^2}{16\sigma^2 n}}
<
r^{\frac{\eps}{16}}.
$$

It remains to consider the case $t \geq \sqrt r$, $|\sigma| \geq 1/(2A)$.
The density of the random variable $\ee^{\sigma \sqrt n X}$ is given by
$$
g(x)=\frac{1}{\sqrt {2\pi n}\sigma x}\ee^{-\frac{(\log x)^2}{2\sigma^2 n}},
\qquad
x>0.
$$
It attains its maximum at $x_0=\ee^{-\sigma^2 n}$. The maximum is equal to
$g(x_0)=\frac {1}{\sqrt {2\pi n}\sigma} \ee^{\sigma^2 n/2}$. Let $r \leq (2\pi
n)\sigma^2 \ee^{-\sigma^2 n}$. Then,
$$
\P[\ee^{\sqrt n (\sigma X+ i\tau Y)} \in B_r(t)]
\leq
\P[t-r\leq \ee^{\sigma \sqrt n X}\leq t+r]\leq \frac {Cr}{\sqrt {n}\sigma} \ee^{\sigma^2 n/2} \leq  C r^{1/2}.
$$
Let $r \geq (2\pi n)\sigma^2 \ee^{-\sigma^2 n}$, which, together with
$|\sigma|>1/(2A)$ and $n>(2A)^2$, implies that $r>\ee^{-\sigma^2 n}$. 
Using the unimodality of the density $g$ and the inequality $t-r>r$, we get
$$
\P[\ee^{\sqrt n (\sigma X+ i\tau Y)} \in B_r(t)]
\leq \P[t-r\leq \ee^{\sigma \sqrt n X}\leq t+r]
< 2r g(r)
< \ee^{-\frac {(\log r)^2}{2\sigma^2 n}}
< r^{\frac{\eps}{2}}.
$$
The last inequality follows from $r<\ee^{-\sigma^2 n}$. This completes the proof of~\eqref{eq:bound_density_small_r}.

Now we are in position to complete the proof of~\eqref{eq:Lq_bound_small_pN}.
Let $U_r$ be a random variable distributed uniformly on the disc $B_r(0)$
and independent of all variables considered previously. It follows
from~\eqref{eq:bound_density_small_r} that the density of the random variable
$\ee^{\sqrt n (\sigma X+ i\tau Y)}+U_r$ is bounded from above by $Cr^{-2+(\eps/20)}$.
Hence, the density of $\ZZZ_N(\beta)+U_r$ is bounded by the same term
$Cr^{-2+(\eps/20)}$. With the notation $r=\ee^{-kn}$ it follows that, for every
$k\geq \eps \sigma^2$,
$$
\P[p_N(\beta) \leq -k]
=
\P[|\ZZZ_N(\beta)|\leq \ee^{-kn}]
\leq
\P[|\ZZZ_N(\beta)+U_{r}| \leq 2r ]
\leq
Cr^{\frac{\eps}{20}}.
$$
From this, we obtain that
$$
\E [|p_N(\beta)|^q \ind_{p_N(\beta)\in [-k-1,-k]}]
\leq
C (k+1)^{q} \ee^{-\frac{\eps kn}{20}}.
$$
Taking the sum over all $k=\eps \sigma^2+l$, $l=0,1,\ldots$, we get
$$
\E [|p_N(\beta)|^q \ind_{p_N(\beta)<-\eps \sigma^2}]
\leq
C \ee^{-\frac{\eps^2 \sigma^2 n}{20}} \sum_{l=1}^{\infty}  l^{q} \ee^{-\frac{\eps ln}{20}}
\leq
C \ee^{-\frac{\eps^2 \sigma^2 n}{20}}.
$$
Recalling~\eqref{eq:Lq_bound_small_pN_step1}, we arrive at~\eqref{eq:Lq_bound_small_pN}.
\end{proof}

\begin{remark}\label{rem:proof_zeros_B3_dominated_conv}
As a byproduct of the proof, we have the following statement. For every
$A>0$, there is a constant $C=C(A)$ such that $\E |p_N(\beta)| < C$, for all
$1/A<|\beta|<A$ and sufficiently large $N$.
\end{remark}

\section{Proofs of the results on zeros}\label{sec:proofs_zeros}
\subsection{Convergence of random analytic functions}

In this section, we collect some lemmas on weak convergence of stochastic
processes whose sample paths are analytic functions. As we will see, the
analyticity assumption simplifies the things considerably. For a metric space
$M$, denote by $C(M)$ the space of complex-valued continuous functions on $M$
endowed with the  topology of uniform convergence on compact sets. Let $D\subset
\C$ be a simply connected domain.

\begin{lemma}\label{lem:random_analytic_ineq}
Let $\{U(t)\colon t\in D\}$ be a random analytic function defined on $D$. Let
$\Gamma\subset D$  be a closed differentiable contour and let $K$ be a compact
subset located strictly inside $\Gamma$.  Then, for every $p\in\N_0$, there is a
constant $C=C_p(K,\Gamma)$ such that
$$
\E \left[\sup_{t\in K} |U^{(p)}(t)|\right] \leq C \oint_{\Gamma} \E |U(w)| |\dd w|.
$$
\end{lemma}
\begin{proof}
By the Cauchy formula, $|U^{(p)}(t)|\leq C \oint_{\Gamma} |U(w)| |\dd w|$, for all $t\in K$. Take the supremum over $t\in K$ and then the expectation.
\end{proof}
It is easy to
check that a sequence of stochastic processes with paths in $C(D)$ is tight (resp., weakly convergent) if
and only if it is tight (resp., weakly convergent) in  $C(K)$, for every compact set $K\subset D$.
\begin{lemma}\label{lem:random_analytic_weak_conv}
Let $U_1,U_2,\ldots$ be random analytic functions on $D$.
Assume that there is a continuous function $f\colon D\to \R$ such that  $\E |U_N(t)| < f(t)$, for all $t\in D$, and all $N\in\N$.
Then, the sequence $U_N$ is tight on $C(D)$.
\end{lemma}
\begin{proof}
Let $K\subset D$ be a compact set. Let $\Gamma$ be a contour enclosing $K$ and located inside $D$. By Lemma~\ref{lem:random_analytic_ineq},
$$
\E \left[\sup_{t\in K} |U_N(t)|\right]
\leq
C \oint_{\Gamma} f(w) |\dd w|,
\;\;\;
\E \left[\sup_{t\in K} |U_N'(t)|\right]
\leq
C \oint_{\Gamma} f(w) |\dd w|.
$$
By standard arguments, this implies that the sequence $U_N$ is tight on $C(K)$.
\end{proof}

\begin{lemma}\label{lem:random_analytic_point_proc}
Let $U,U_1,U_2,\ldots$ be random analytic functions on $D$ such that $U_N$
converges as $N\to\infty$ to $U$ weakly on $C(D)$ and $\P[U\equiv 0]=0$.
 Then, for every continuous function $f:D\to\R$ with compact support, we have
$$
\sum_{z\in\C \colon U_N(z)=0} f(z) \todistr  \sum_{z\in\C \colon U(z)=0} f(z).
$$
\end{lemma}
\begin{remark}
Equivalently, the zero set of $U_N$, considered as a point process on $D$, converges weakly to the zero set of $U$.
\end{remark}
\begin{proof}
Let $H$ be the closed linear subspace of $C(D)$ consisting of all analytic
functions. Consider a functional $\Psi \colon H\to \R$ mapping an analytic function
$\varphi$ which is not identically $0$ to $\sum_{z} f(z)$, where the sum is over
all zeros of $\varphi$. Define  also $\Psi(0)=0$. It is an easy consequence of
Rouch\'e's theorem that $\Psi$ is continuous on $H\bsl\{0\}$. Note that $H\bsl\{0\}$ is a set of
full measure with respect to the law of $U$.  Recall that $U_N\to U$ weakly on
$H$. By the continuous mapping theorem~\cite[\S3.5]{resnick_book}, $\Psi(U_N)$
converges in distribution to $\Psi(U)$. This proves the lemma.
\end{proof}

\subsection{Proof of Theorem~\ref{theo:zeros_in_B3}}
A standard approximation argument shows that we can assume that $f$ is infinitely differentiable. Let $\lambda$ be the Lebesgue measure on $\C$.
In his computation of the limiting density of zeros, \citet{DerridaComplexREM1991} used the fact that $\frac 1 {2\pi}\Delta \log |\ZZZ_N|$ (where $\Delta$ is the Laplacian interpreted in the distributional sense) gives the measure counting the zeros of $\ZZZ_N$. That is,
\begin{equation}\label{eq:laplace_log_zeros}
\sum_{\beta\in\C \colon \ZZZ_N(\beta)=0} f(\beta)
=
\frac 1 {2\pi}
\int_{\C} \log |\ZZZ_N(\beta)| \Delta f(\beta) \lambda(\dd \beta).
\end{equation}
A proof of~\eqref{eq:laplace_log_zeros} can be found in~\cite[Section~2.4]{peres_etal_book}. Recall that $p(\beta)$ has been defined in Theorem~\ref{theo:log_scale}. We have
\begin{equation}\label{eq:distr_laplace_p}
\int_{\C} p(\beta) \Delta f(\beta) \lambda(\dd\beta) = \int_{\C} f(\beta) \Xi(\dd\beta).
\end{equation}
Indeed, Green's  identity gives
$$
\int_{B_i} p(\beta) \Delta f(\beta) \lambda(\dd \beta)
=
\int_{B_i} \Delta p(\beta) f(\beta) \lambda(\dd \beta)
+
\oint_{\partial B_i} \left(p(\beta)\frac{\partial f(\beta)}{\partial \nn} - f(\beta)\frac{\partial p(\beta)}{\partial \nn}\right) |\dd \beta|.
$$
Here, $\nn$ is the unit normal to the boundary of $B_i$ pointing outwards $B_i$
and $\frac {\partial}{\partial \nn}$ is the corresponding directional
derivative. The first term on the right-hand side is equal to $2\int_{\C}
f(\beta) \Xi_3(\dd \beta)$ for $i=3$ and to $0$ for $i=1,2$. Adding Green's
identities for $i=1,2,3$, noting that $\frac{\partial f}{\partial \nn}$ has no
jumps and computing the jumps of $\frac{\partial p}{\partial \nn}$ on the
boundaries between the different $B_i$'s, we arrive
at~\eqref{eq:distr_laplace_p}.

Recall that $p_N(\beta)=\frac 1n \log |\ZZZ_N(\beta)|$. From~\eqref{eq:laplace_log_zeros} and~\eqref{eq:distr_laplace_p} we conclude that Theorem~\ref{theo:zeros_in_B3} is equivalent to
$$
\int_{\C}  p_N(\beta) \Delta f(\beta) \lambda(\dd \beta) \toprobab \int_{\C} p(\beta)  \Delta f(\beta)\lambda(\dd \beta).
$$
We will show that this holds even in $L^1$.  By  Fubini's theorem, it suffices to show that
\begin{equation}\label{eq:conv_zeros_L1}
\lim_{N\to\infty}   \int_{\C}  \E |p_N(\beta)-p(\beta)| |\Delta f(\beta)| \lambda(\dd \beta) = 0.
\end{equation}
We know from Theorem~\ref{theo:log_scale} that $\lim_{N\to\infty} \E
|p_N(\beta)-p(\beta)|=0$, for every $\beta\in \C$. To complete the proof, we
need to interchange the limit and the integral. We may represent $f$ as a sum of
two functions, the first one vanishing on $|\beta|<1/4$ and the second one
vanishing outside $|\beta|<1/2$. Since the contribution of the second function
to~\eqref{eq:theo:zeros_in_B3} vanishes by Theorem~\ref{theo:no_zeros_B1}, we
may assume that $f$ vanishes on $|\beta|<1/4$. With this assumption, the use of
the dominated convergence  theorem is justified by
Remark~\ref{rem:proof_zeros_B3_dominated_conv}. \hfill $\Box$

\subsection{Proof of Theorem~\ref{theo:no_zeros_B1}} The idea of the proof is to
show that the fluctuations of $\ZZZ_N(\beta)$ around its expectation are of
smaller order than the expectation, in phase $B_1$. We don't rely on the
expression for the limiting log-partition function $p$.   Let $\Gamma$ be a
differentiable contour enclosing the set $K$ and located inside $B_1$. We have
\begin{align*}
\P[\ZZZ_N(\beta)=0, \text{ for some } \beta\in K]
&\leq
\P\left[\sup_{\beta\in K} \left|\frac{\ZZZ_N(\beta)-\E \ZZZ_N(\beta)}{\E \ZZZ_N(\beta)}\right|\geq 1\right]\\
&\leq
\E \sup_{\beta\in K} \left|\frac{\ZZZ_N(\beta)-\E \ZZZ_N(\beta)}{\E \ZZZ_N(\beta)}\right|\\
&\leq
C \oint_{\Gamma} \E \left|\frac{\ZZZ_N(\beta)-\E \ZZZ_N(\beta)}{\E \ZZZ_N(\beta)}\right| |\dd w|,
\end{align*}
where the last step is by Lemma~\ref{lem:random_analytic_ineq}. Note that $|\E
\ZZZ_N(\beta)|=N^{1+\frac 12 (\sigma^2-\tau^2)}$. To complete the proof, we need
to show that there exist  $\eps>0$ and $C>0$ depending on $\Gamma$ such that, for
every $\beta\in \Gamma$, $N\in\N$,
\begin{equation}\label{eq:B1_proof_need}
\E \left|\ZZZ_N(\beta)-\E \ZZZ_N(\beta)\right| < C N^{1-\eps+\frac 12 (\sigma^2-\tau^2)}.
\end{equation}
Since $\Gamma\subset B_1$, we can choose $\eps>0$ so small that $\Gamma\subset B_1'(\eps)\cup B_1''(\eps)$, where
\begin{align*}
B_1'(\eps)&=\{\beta\in\C \colon \sigma^2+\tau^2<1-2\eps\},\\
B_1''(\eps)&=\{\beta\in\C \colon (|\sigma|-\sqrt 2)^2-\tau^2>2\eps, 1/2<\sigma^2<2\}.
\end{align*}
We have
$$
\E|\ZZZ_N(\beta)-\E \ZZZ_N(\beta)|^2 = N \E |\ee^{\beta\sqrt n X}- \E \ee^{\beta\sqrt n X}|^2 \leq
N\E \ee^{2\sigma\sqrt n X}= N^{1+2\sigma^2}.
$$
If $\beta\in B_1'(\eps)$, then it follows that
$$
\E \left|\ZZZ_N(\beta)-\E \ZZZ_N(\beta)\right|
\leq
N^{\frac 12 +\sigma^2}
\leq
 N^{1-\eps+\frac 12 (\sigma^2-\tau^2)}.
$$
This implies~\eqref{eq:B1_proof_need}. Assume now that $\beta\in B_1''(\eps)$ and $\sigma>0$.
For $k=1,\ldots,N$, define random variables
$$
U_{k,N}=\ee^{\beta \sqrt n X_k-\sigma \sqrt {2}n} \ind_{X_k \leq \sqrt {2n}},
\;\;\;
V_{k,N}=\ee^{\beta \sqrt n X_k-\sigma \sqrt {2}n} \ind_{X_k > \sqrt {2n}},
$$
By Part~\ref{lem:saddle_weak1} of Lemma~\ref{lem:saddle_weak}, we have
\begin{equation}\label{eq:B1_proof_eq1}
\E \left| \sum_{k=1}^N (U_{k,N}-\E U_{k,N})\right|^2
\leq
N\E |U_{1,N}|^2
=
N\ee^{-2\sqrt 2 \sigma n} \E [\ee^{2\sigma \sqrt n X} \ind_{X<\sqrt {2n}}]
<1.
\end{equation}
Similarly, by Part~\ref{lem:saddle_weak2} of Lemma~\ref{lem:saddle_weak},
\begin{equation}\label{eq:B1_proof_eq2}
\E \left|\sum_{k=1}^N (V_{k,N} - \E V_{k,N})\right|
\leq
2N\E |V_{k,N}|
=
2N \ee^{-\sigma\sqrt 2 n} \E [\ee^{\sigma \sqrt n X} \ind_{X>\sqrt {2n}}]
<
2.
\end{equation}
Combining~\eqref{eq:B1_proof_eq1} and~\eqref{eq:B1_proof_eq2}, we obtain $\E
\left|\ZZZ_N(\beta)-\E \ZZZ_N(\beta)\right|\leq 3\ee^{\sigma \sqrt {2}n}$. Since
$\beta\in B_1''(\eps)$, this implies the required
estimate~\eqref{eq:B1_proof_need}. \hfill $\Box$

As a by-product, we obtain a  proof of the formula for the limiting
log-partition function in phase $B_1$ which is simpler than the proof given in
Section~\ref{subsec:proof_theo_log_scale}. We will use only the information
about the first two truncated moments of $\ZZZ_N(\beta)$.
\begin{proposition}
For $\beta\in B_1$, we have $\lim_{N\to\infty} p_N(\beta)=1+\frac 12 (\sigma^2-\tau^2)$ in probability.
\end{proposition}
\begin{proof}
We have shown in~\eqref{eq:B1_proof_need} that, for every $\beta\in B_1$, there
exist $C>0$ and $\eps>0$ depending on $\beta$ such that for all $N\in\N$,
\begin{equation}\label{eq:tech24}
\E \left|\frac {\ZZZ_N(\beta)}{\E \ZZZ_N(\beta)}-1\right|<C N^{-\eps}.
\end{equation}
Let $p(\beta)=1+\frac 12 (\sigma^2-\tau^2)$. Note that $\frac 1n \log |\E
\ZZZ_N(\beta)|=p(\beta)$ for all $N\in\N$.  It follows that for every $\delta>0$
and all sufficiently large $N$,
\begin{equation}\label{eq:tech24a}
\P[|p_N(\beta)-p(\beta)|>\delta]
=
\P\left[\left|\log \left|\frac{\ZZZ_N(\beta)}{\E \ZZZ_N(\beta)}\right|\right| > n \delta \right]
\leq
\P\left[\left|\frac{\ZZZ_N(\beta)}{\E \ZZZ_N(\beta)}-1\right| > \frac 12 \right].
\end{equation}
In the last step, we have used that $|\log |z||>n\delta$ implies that
$|z-1|>1/2$, for large $n$. By~\eqref{eq:tech24} and the Markov inequality we
can estimate the right-hand side of~\eqref{eq:tech24a} by $2CN^{-\eps}$, which
implies that $p_N(\beta)$ converges to $p(\beta)$ in probability.
\end{proof}

\subsection{Proof of Theorems~\ref{theo:local_zeros_small_beta}
and~\ref{theo:boundary_zeros1}} Recall that $\GGG$ is the Gaussian analytic
function defined in~\eqref{eq:def_GGG}. 
Theorem~\ref{theo:local_zeros_small_beta} will be deduced from the following
result.
\begin{theorem}\label{theo:weak_conv_small_beta}
Fix some $\beta_0=\sigma_0+i\tau_0$ with $\sigma_0^2<1/2$ and $\tau_0\neq 0$.
Define a random process $\{G_N(t)\colon t\in\C\}$ by
\begin{equation}\label{eq:def_GN}
G_N(t):=\frac{\ZZZ_N\left(\beta_0+\frac{t}{\sqrt n}\right)- N^{1+\frac 12 (\beta_0+\frac t {\sqrt n})^2}}{N^{\frac 12+ (\sigma_0+\frac t {\sqrt n})^2}}.
\end{equation}
Then,  the process $G_N$ converges weakly, as $N\to\infty$, to the process
$\ee^{-t^2/2}\GGG(t)$ on $C(\C)$.
\end{theorem}
\begin{proof}
For $k=1,\ldots,N$, define a random process $\{W_{k,N}(t)\colon t\in\C\}$ by
$$
W_{k,N}(t) = N^{-1/2} \ee^{(\beta_0\sqrt n +t)X_k - (\sigma_0 \sqrt n + t)^2}.
$$
Then, $ G_N(t)=\sum_{k=1}^N (W_{k,N}(t)-\E W_{k,N}(t)) $. First, we show that
the convergence stated in Theorem~\ref{theo:weak_conv_small_beta} holds in the
sense of finite-dimensional distributions. Take $t_1,\ldots,t_d\in\C$. Write
$\WWW_{k,N}=(W_{k,N}(t_1), \ldots, W_{k,N}(t_d))$.  We need to prove that
\begin{equation}\label{eq:conv_vectors_G}
\sum_{k=1}^N (\WWW_{k,N}-\E \WWW_{k,N})
\todistr
(\ee^{-t_1^2/2}\GGG(t_1),\ldots, \ee^{-t_d^2/2}\GGG(t_d)).
\end{equation}
Let  $W_N$ be a process having the same law as the $W_{k,N}$'s and define
$\WWW_N=(W_N(t_1),\ldots,W_N(t_d))$.  A straightforward computation shows that
for all $t,s\in\C$,
\begin{align}
N\E[W_N(t)\overline{W_N(s)}]
&=
\ee^{-(t-\overline{s})^2/2},\label{eq:conv_WN_proc_1}
\\
\lim_{N\to\infty} N|\E[W_N(t)W_N(s)]|
&=
 0. \label{eq:conv_WN_proc_2}
\end{align}
Also, we have
\begin{equation}\label{eq:conv_WN_proc_3}
\lim_{N\to\infty} \sqrt N |\E [W_N(t)]|
=
\lim_{N\to\infty} \ee^{-\frac 12 (\sigma_0^2+\tau_0^2) n +O(\sqrt n)}
=
0.
\end{equation}
Note that by~\eqref{eq:def_GGG},
$$
\E [\ee^{-t^2/2} \GGG(t)\overline{\ee^{-s^2/2}\GGG(s)}]=\ee^{-(t-\overline{s})^2/2},
\qquad
\E [\ee^{-t^2/2} \GGG(t) \ee^{-s^2/2}\GGG(s)]=0.
$$
We see that the covariance matrix of the left-hand side
of~\eqref{eq:conv_vectors_G} converges to the covariance matrix of the
right-hand side of~\eqref{eq:conv_vectors_G} if we view both sides as
$2d$-dimensional real random vectors.  To complete the proof
of~\eqref{eq:conv_vectors_G},  we need to verify the Lindeberg condition: for
every $\eps>0$,
\begin{equation}\label{eq:lindeberg_multidim}
\lim_{N\to\infty} N \E [|\WWW_N|^2 \ind_{|\WWW_N|>\eps}]=0.
\end{equation}
 For $l=1,\ldots,d$, let $A_l$ be the random event $|W_{N}(t_l)|\geq
|W_{N}(t_j)|$ for all $j=1,\ldots,d$. On $A_l$, we have $|\WWW_N|^2\leq d
|W_N(t_l)|^2$. It follows that
$$
N \E [|\WWW_N|^2 \ind_{|\WWW_N|>\eps}]
\leq
 d  \sum_{l=1}^d N \E \left[|W_N(t_l)|^2 \ind_{|W_N(t_l)|>\frac{\eps}{\sqrt d}}\right]
\to 0,
$$
where the last step is by the same argument as in~\eqref{eq:proof_lindeberg}.
This completes the proof of the finite-dimensional convergence stated
in~\eqref{eq:conv_vectors_G}. The tightness follows from
Lemma~\ref{lem:random_analytic_weak_conv} which can be applied since
$$
\E |G_N(t)|\leq  \sqrt{\E[|G_N(t)|^2]} \leq \sqrt{N \E [|W_N(t)|^2]}= \ee^{(\Im t)^2}.
$$
The last equality follows from~\eqref{eq:conv_WN_proc_1}.
\end{proof}
\begin{proof}[Proof of Theorem~\ref{theo:local_zeros_small_beta}]
If   $\beta_0\in B_3$, then the expectation term in the definition of $G_N$, see~\eqref{eq:def_GN}, can
be ignored: we have $\lim_{N\to\infty} |G_N(t)-U_N(t)|=0$ uniformly on compact
sets, where
$$
U_N(t)=N^{-\frac 12-(\sigma_0+\frac t {\sqrt n})^2} \ZZZ_N\left(\beta_0+\frac{t}{\sqrt n}\right).
$$
It follows from Theorem~\ref{theo:weak_conv_small_beta} that $U_N$ converges to the process
$\ee^{-t^2/2} \GGG(t)$ weakly on $C(\C)$.
 Applying Lemma~\ref{lem:random_analytic_point_proc}, we obtain the statement of
 Theorem~\ref{theo:local_zeros_small_beta}.
\end{proof}

\begin{proof}[Proof of Theorem~\ref{theo:boundary_zeros1}]
Let $\delta_N$ be a bounded sequence such that $n\sigma_0\tau_0 - \delta_N\in
2\pi \Z$.  Taking $t=\beta_0 \frac {s+ i \delta_N} {\sqrt n}$ in
Theorem~\ref{theo:weak_conv_small_beta}, we obtain that weakly on $C(\C)$,
$$
G_N\left(\beta_0\frac {s+i \delta_N} {\sqrt n}\right) \toweak \GGG(0).
$$
Doing elementary transformations, we arrive at
$$
\frac{\ZZZ_N\left(\beta_0\left(1+\frac {s+ i \delta_N} n\right)\right)}{N^{\frac 12 + (\sigma_0+\beta_0\frac {s+ i \delta_N } n)^2}} \toweak  \ee^{-s}+\GGG(0).
$$
The zeros of the right-hand side are located at $s = 2 \pi i k +\xi$, $k\in\Z$,
where $\xi=-\log (-\GGG(0))$. The proof is completed by applying
Lemma~\ref{lem:random_analytic_point_proc}.
\end{proof}

\subsection{Proof of Theorems~\ref{theo:zeta_poi_exists}, \ref{theo:local_zeros_large_beta} and~\ref{theo:boundary_zeros2}}
\begin{proof}[Proof of Theorem~\ref{theo:zeta_poi_exists}]
Fix a compact set $K$ contained in the half-plane $\sigma>1/2$. Define  random
$C(K)$-valued elements $S_k(\beta)=s_1(\beta)+\ldots+s_k(\beta)$, where
$$
s_k(\beta)=\sum_{j=1}^{\infty} P_j^{-\beta} \ind_{k\leq P_j<k+1} -\int_{k}^{k+1} t^{-\beta} \dd t, \;\;\; \beta\in K.
$$
Note that $s_1,s_2,\ldots$ are independent.
By the properties of the Poisson process,
\begin{equation}\label{eq:zeta_campbell}
\E [s_k(\beta)]=0,\;\;\;
\sum_{k=1}^{\infty} \E [|s_k(\beta)|^2]=\int_{1}^{\infty} t^{-2\sigma} \dd t < \infty.
\end{equation}
Thus, as long as $\sigma>1/2$, the sequence $\{S_k(\beta) \}_{k\in\N}$ is an
$L^2$-bounded  martingale.  Hence, $S_k(\beta)$ converges a.s.\ to a limiting
random variable denoted by $S(\beta)$.  We need to show that the convergence is
uniform a.s. It follows from~\eqref{eq:zeta_campbell} and
Lemma~\ref{lem:random_analytic_weak_conv} that the sequence $S_k$, $k\in\N$, is tight
on $C(K)$. Hence, $S_k$ converges weakly on $C(K)$ to the process $S$.  By the
It\^{o}--Nisio theorem~\cite{ito_nisio}, this implies that $S_k$ converges to
$S$ a.s.\ as a random element of $C(K)$. This proves the theorem.
\end{proof}
\begin{proof}[Proof of Theorem~\ref{theo:local_zeros_large_beta}]
Let us first describe the idea. Consider the case $\sigma>1/\sqrt 2$. Arrange
the values $X_1,\ldots, X_{N}$ in an increasing order, obtaining the  order
statistics $X_{1:N}\leq\ldots \leq X_{N:N}$. It turns out that the main
contribution to the sum $\ZZZ_N(\beta)=\sum_{k=1}^{N} \ee^{\beta \sqrt n X_k}$
comes from the upper order statistics $X_{N-k:N}$, where $k=0,1,\ldots$. Their
joint limiting distribution is well-known in the extreme-value theory,
see~\cite[Corollary~4.19(i)]{resnick_book}, and will be recalled now.  Denote by
$\MMM$ the space of locally finite counting measures on $\bar
\R=\R\cup\{+\infty\}$. We endow $\MMM$ with the (Polish) topology of vague
convergence. A point process on $\bar\R$ is a random element with values in
$\MMM$. Let $P_1,P_2,\ldots$  be the arrivals of the unit intensity Poisson
process on the positive half-line. Define the sequence $b_N$ as
in~\eqref{eq:def_wN}, that is
$$
b_N=\sqrt {2n}- \frac {\log (4\pi n)}{2\sqrt {2 n}}.
$$
\begin{proposition}\label{prop:resnick}
The point process $\pi_N:=\sum_{k=1}^N \delta(\sqrt n(X_k-b_N))$ converges as
$N\to\infty$ to the point process $\pi_{\infty}=\sum_{k=1}^{\infty}
\delta(-(\log P_k)/\sqrt 2)$ weakly on $\MMM$.
\end{proposition}
Utilizing this result, we will show that it is possible to approximate
$\ZZZ_N(\beta)$ (after appropriate normalization) by $\tilde \zeta_P(\beta/\sqrt
2)$ in the half-plane $\sigma>1/\sqrt 2$. Consider now the case $\sigma<-1/\sqrt
2$. This time, the main contribution to the sum $\ZZZ_N(\beta)$ comes from the
lower order statistics $X_{k:N}$, $k=1,2,\ldots$ Their joint limiting
distribution is the same as for the upper order statistics, only the sign should
be reversed.  Moreover, it is known that the upper and the lower order
statistics become asymptotically independent as $N\to\infty$; see~\cite{ikeda}
or~\cite[Cor. 5.28]{resnick_book}. Thus, in the half-plane $\sigma<-1/\sqrt 2$
it is possible to approximate $\ZZZ_N(\beta)$ by an independent copy of
$\zeta_P(-\beta/\sqrt 2)$. In the rest of the proof we make this idea rigorous.
For simplicity of notation we restrict ourselves to the half-plane
$D=\{\beta\in\C \colon \sigma>1/\sqrt 2\}$.

\begin{theorem}\label{theo:weak_conv_large_beta}
The following convergence holds weakly on $C(D)$:
$$
\xi_N(\beta):=\frac{\ZZZ_N(\beta)-N\E [\ee^{\beta \sqrt n X}1_{X<b_N}]}{\ee^{\beta \sqrt n b_N}}
\toweak
\tilde \zeta_{P}\left(\frac{\beta}{\sqrt 2}\right).
$$
\end{theorem}
The proof consists of two lemmas. Take $A>0$ and write $\xi_N(\beta)=\xi_N^A(\beta)-e_N^A(\beta)+\Delta_N^A(\beta)$, where
\begin{align*}
\xi_N^A(\beta)&=\sum_{k=1}^N \ee^{\beta \sqrt n (X_k-b_N)}\ind_{b_N-\frac A{\sqrt n}<X_k},\\
e_N^A(\beta)&=N\E\left[\ee^{\beta \sqrt n (X_k-b_N)}\ind_{b_N-\frac A{\sqrt n}\leq X_k<b_N}\right],\\
\Delta_N^A(\beta)&=\sum_{k=1}^N \left( \ee^{\beta \sqrt n (X_k-b_N)}\ind_{X_k \leq  b_N - \frac A{\sqrt n}} - \E\left[\ee^{\beta \sqrt n (X_k-b_N)}\ind_{X_k\leq b_N-\frac A{\sqrt n}}\right]\right).
\end{align*}
\begin{lemma}\label{lem:proof_conv_to_zeta_1}
Let $\tilde \zeta_P(\cdot;\cdot)$ be defined as in~\eqref{eq:zeta_poi_tilde}.
Then, the following convergence holds weakly on $C(D)$:
$$
\xi_N^A(\beta)-e_N^A(\beta)\toweak  \tilde \zeta_{P}\left(\frac{\beta}{\sqrt 2}; \ee^{\sqrt 2A}\right).
$$
\end{lemma}
\begin{proof}
Recall that by Proposition~\ref{prop:resnick} the point process $\pi_N$
converges to the point process $\pi_{\infty}$ weakly on $\MMM$. Consider a
functional $\Psi\colon \MMM \to C(D)$ which maps a locally finite counting measure
$\rho=\sum_{i\in I}\delta(y_i)\in \MMM$ to the function
$\Psi(\rho)(\beta)=\sum_{i\in I} \ee^{\beta y_i}\ind_{y_i>-A}$, where $\beta\in
D$. Here, $I$ is at most countable index set. If $\rho$ charges the point
$+\infty$, define $\Psi(\rho)$, say, as $0$.  The functional $\Psi$ is
continuous on the set of all  $\rho\in \MMM$ not charging the points $-A$ and
$+\infty$, which is a set of full measure with respect to the law of
$\pi_{\infty}$. It follows from the continuous mapping
theorem~\cite[\S3.5]{resnick_book} that $\xi_N^A=\Psi(\pi_N)$ converges weakly
on $C(D)$ to $\Psi(\pi_{\infty})$. Note that
$$
\Psi(\pi_{\infty})(\beta)=\sum_{k=1}^{\infty}P_k^{-\beta/\sqrt 2}\ind_{P_k<\ee^{\sqrt 2 A}}.
$$
We prove the convergence of $e_N^A(\beta)$. Using the change of variables $\sqrt
n(x-b_N)=y$, we obtain
\begin{align*}
e_N^A(\beta)
=
\frac{N}{\sqrt {2\pi}} \int_{b_N-\frac{A}{\sqrt n}}^{b_N} \ee^{\beta \sqrt n (x-b_N)} \ee^{-\frac{x^2}{2}}\dd x
=
\frac{N}{\sqrt {2\pi n}} \int_{-A}^{0} \ee^{\beta y} \ee^{-\frac 12 (b_N+\frac{y}{\sqrt n})^2}\dd y.
\end{align*}
Recalling that $\sqrt{2\pi}b_N \ee^{b_N^2/2}\sim N$ and $b_N\sim \sqrt {2n}$ as $N\to\infty$, we obtain that $\lim_{N\to\infty}e_N^A(\beta)=\int_{1}^{\ee^{\sqrt 2 A}} t^{-\beta/\sqrt 2} \dd t$, as required.
\end{proof}
\begin{lemma}\label{lem:proof_conv_to_zeta_2}
For every compact set $K\subset D$, there is $C>0$ such that, for all sufficiently large $N$,
$$
\E\left[\sup_{\beta\in K} |\Delta_N^A(\beta)|\right]
\leq C \ee^{(1-\sqrt 2 \sigma) A/2}.
$$
\end{lemma}
\begin{proof}
Let $\Gamma$ be a contour enclosing $K$ and located inside $D$. First, $\E
[\Delta_N^A(\beta)]=0$ by definition. Second, uniformly in $\beta\in \Gamma$  it
holds that
\begin{align*}
\E [|\Delta_N^A(\beta)|^2]
&\leq
N \E \left[\ee^{2\sigma\sqrt n (X-b_N)}\ind_{X<b_N-\frac{A}{b_N}} \right]\\
&=
N \ee^{-2\sigma \sqrt n b_N} \ee^{2\sigma^2 n} \Phi\left(b_N-\frac{A}{b_N}-2\sigma \sqrt n\right)\\
&\leq
C \ee^{(1-\sqrt 2 \sigma) A},
\end{align*}
where the second step follows from Lemma~\ref{lem:trunc_moment_Phi} and the last
step follows from~\eqref{eq:Phi_asympt}. By Lemma~\ref{lem:random_analytic_ineq},
we have
$$
\E\left[\sup_{\beta\in K} |\Delta_N^A(\beta)|\right]
\leq
C \oint_{\Gamma} \E |\Delta_N^A(\beta)| |\dd \beta|
\leq
C \ee^{(1-\sqrt 2 \sigma) A/2}.
$$
The proof is complete.
\end{proof}
\begin{proof}[Proof of Theorem~\ref{theo:weak_conv_large_beta}]
By Theorem~\ref{theo:zeta_poi_exists}, we have the weak convergence
$$
\tilde \zeta_P\left(\frac{\beta}{\sqrt 2}; \ee^{\sqrt 2 A}\right)
\overset{d}{\underset{A\to\infty}\longrightarrow}
\tilde \zeta_P\left(\frac{\beta}{\sqrt 2}\right).
$$
Together with Lemmas~\ref{lem:proof_conv_to_zeta_1}
and~\ref{lem:proof_conv_to_zeta_2}, this implies
Theorem~\ref{theo:weak_conv_large_beta} by a standard argument; see for
example~\cite[Lemma~6.7]{kabluchko1}.
\end{proof}

The proof of Theorem~\ref{theo:local_zeros_large_beta} can be completed as
follows. For $\sigma>1/\sqrt 2$, Lemma~\ref{lem:saddle_point} yields
$$
\lim_{N\to\infty} N \ee^{-\beta \sqrt n b_N} \E [\ee^{\beta \sqrt n X}1_{X<b_N}]
=
\begin{cases}
\frac{\sqrt 2}{\beta-\sqrt 2}, & \text{ if } |\sigma|+|\tau|>\sqrt 2,\\
\infty, & \text{ if } |\sigma|+|\tau| \leq \sqrt 2.
\end{cases}
$$
The first equality holds uniformly on compact subsets of $B_2$. By
Theorem~\ref{theo:weak_conv_large_beta}, the process $\ee^{-\beta \sqrt n
b_N}\ZZZ_N(\beta)$ converges to $\zeta_P(\beta/\sqrt 2)$ weakly on the space of
continuous functions on the set $B_2\cap \{\sigma>1/\sqrt 2\}$. Similarly, on
the space of continuous functions on $B_2\cap \{\sigma<-1/\sqrt 2\}$ the same
process converges weakly to an independent copy of $\zeta_P(-\beta/\sqrt 2)$. By
Lemma~\ref{lem:random_analytic_point_proc}, this implies
Theorem~\ref{theo:local_zeros_large_beta}.
\end{proof}

\begin{proof}[Proof of Theorem~\ref{theo:boundary_zeros2}]
Let $d_N'$ be a complex sequence such that
$$
d_N'+ \beta_0 \frac{\log (4\pi n)}{2\sqrt 2}-i\tau_0^2 n\in 2\pi i \Z
\;\;\;
\text{and}
\;\;\;
d_N'=O(\log n).
$$
Write
$\beta_N=\beta_0+\frac{s+d_N'}{(\beta_0-\sqrt 2)n}$, where $s\in\C$ is a new
variable. Note that $\lim_{N\to\infty}\beta_N=\beta_0$.
Let $X\sim N_{\R}(0,1)$. Applying Remark~\ref{rem:saddle} and noting that the second
term on the right-hand side in~\eqref{eq:lem_saddle_complex_two_terms} converges
to $\frac{\sqrt 2}{\beta_0-\sqrt 2}$, we obtain
\begin{align*}
\lim_{N\to\infty} N \ee^{-\beta_N \sqrt n b_N} \E [ \ee^{\beta_N \sqrt n X} \ind_{X<b_N}]
&=
\lim_{N\to\infty} N \ee^{-\beta_N \sqrt n b_N} \ee^{\frac{\beta_N^2 n} 2 }+\frac{\sqrt 2}{\beta_0-\sqrt 2}\\
&=
\ee^{s} + \frac{\sqrt 2}{\beta_0-\sqrt 2}.
\end{align*}
By
Theorem~\ref{theo:weak_conv_large_beta} and~\eqref{eq:zeta_zeta_tilde}, the following holds weakly on $C(\C)$:
$$
\ee^{-\beta_N \sqrt n b_N} \ZZZ_N\left( \beta_0+\frac{s+d_N'}{(\beta_0-\sqrt 2)n} \right)
\toweak
\ee^{s} + \zeta_P\left(\frac{\beta_0}{\sqrt 2}\right).
$$
The zeros of the right-hand side are located  at $s=2\pi i k + \eta$, $k\in\Z$,
where $\eta=\log (-\zeta_P(\beta_0/\sqrt 2))$. Define $d_N=d_N'/(\sqrt 2
\tau_0)$.  The theorem follows from Lemma~\ref{lem:random_analytic_point_proc}
after elementary transformations.
\end{proof}

\subsection{Proof of Proposition~\ref{prop:zeta_boundary}} Let $\tau\neq 0$ be
fixed.  Let $S(\beta)$ be a random variable defined as in the proof of
Theorem~\ref{theo:zeta_poi_exists}. Take $a,b\in\R$. For $\sigma>1/2$, consider
a random variable
$$
Y(\sigma)
=
a\Re S(\beta)+b \Im S(\beta)
=
\lim_{k\to\infty} \left(\sum_{j=1}^{\infty} f(P_j;\sigma) \ind_{1\leq P_j<k} -\int_{1}^{k} f(t;\sigma) \dd t\right),
$$
where $f(t;\sigma)=\sqrt {a^2+b^2} t^{-\sigma} \cos (\tau \log t - \theta)$ and  $\theta\in\R$ is such that $\cos \theta = \frac {a}{\sqrt{a^2+b^2}}$ and $\sin \theta =  \frac {b}{\sqrt{a^2+b^2}}$.

We need to show that $\sqrt{2\sigma-1}\, Y(\sigma)$ converges, as
$\sigma\downarrow 1/2$, to a centered real Gaussian distribution with variance
$(a^2+b^2)/2$. By the properties of the Poisson process, the log-characteristic
function of $Y(\sigma)$ is given by
$$
\log \E \ee^{iz Y(\sigma)}
=
\int_{1}^{\infty}\left(\ee^{izf(t;\sigma)}-1-izf(t;\sigma)+\frac {z^2}2 f^2(t;\sigma)\right)\dd t
-
\frac {z^2} 2 \int_{1}^{\infty}  f^2(t;\sigma)\dd t.
$$
We will compute the second term and show that the first term is negligible. By elementary integration, we have
\begin{align}
\int_{1}^{\infty} f^2(t;\sigma)\dd t
&=
\frac {a^2+b^2}2 \int_{1}^{\infty} \frac{1+\cos(2\tau\log t-2\theta)}{t^{2\sigma}}\dd t \label{eq:proof_zeta_bound1} \\
&=
\frac {a^2+b^2}2 \left( \frac{1}{2\sigma-1} - \Re\frac{\ee^{-2\theta i}}{(1-2\sigma)+2 i \tau} \right). \notag
\end{align}
Using the inequalities $|\ee^{ix}-1-ix+\frac {x^2}2|\leq |x|^3$ and $|f(t;\sigma)|< C t^{-\sigma}$, we obtain
\begin{equation} \label{eq:proof_zeta_bound2}
\left|\int_{1}^{\infty}\left(\ee^{izf(t;\sigma)}-1-izf(t;\sigma)+\frac {z^2}2 f^2(t;\sigma)\right)\dd t\right|
\leq
\frac{C}{3\sigma-1} |z|^3.
\end{equation}
Bringing~\eqref{eq:proof_zeta_bound1} and~\eqref{eq:proof_zeta_bound2} together and recalling that $\tau\neq 0$, we arrive at
\begin{equation} \label{eq:proof_zeta_bound3}
\lim_{\sigma\downarrow 1/2} \log \E \ee^{i \sqrt {2\sigma-1}\, z Y(\sigma)}
=
-\frac 14 (a^2+b^2) z^2.
\end{equation}
This proves the result for $\tau\neq 0$. For $\tau=0$, the limit
is~\eqref{eq:proof_zeta_bound3} is $-a^2z^2/2$. \hfill $\Box$

 \subsection*{Acknowledgments.} We are grateful to Mikhail~Sodin and Avner~Kiro
for pointing out a mistake in the first version of the paper. We also thank the
unknown referee for useful comments, as well as Dmitry~Zaporozhets for useful
discussions. AK thanks the Institute of Stochastics of Ulm University for kind
hospitality.

\bibliographystyle{plainnat}

\end{document}